\documentclass[10pt,a4paper,oneside,reqno]{amsart}

\usepackage{graphicx}
\usepackage{amssymb}
\usepackage{geometry}
\usepackage{adjustbox}
\usepackage{float}
\usepackage[normalem]{ulem} 
\usepackage{cancel}
\usepackage{hyperref}
\def\R{\mathbb{R}}

\def\Z{\mathbb{Z}}

\def\d{\mathrm{d}}
\newcommand{\dsum}{\displaystyle\sum}

\newcommand{\dmax}{\displaystyle\max}

\usepackage{pgfplots}
\pgfplotsset{compat=1.14}

\newtheorem{lem}{Lemma}[section]
\newtheorem{prop}{Proposition}[section]

\usepackage[ruled,vlined]{algorithm2e}
\usepackage{ulem}
\usetikzlibrary{shapes}
\usetikzlibrary{plotmarks}
\usepackage{array}
\definecolor{armygreen}{rgb}{0.19, 0.53, 0.43}

\usepackage{multicol}
\usepackage{multirow}

\let\origmaketitle\maketitle
\def\maketitle{
  \begingroup
  \def\uppercasenonmath##1{} 
  \let\MakeUppercase\relax 
  \origmaketitle
  \endgroup
}

\begin{document}

\title[Hub Location Problems with Neighborhoods]{\huge On hub location problems in geographically flexible networks}
\date{\today}

\author[V. Blanco \MakeLowercase{and} J. Puerto]{{\large V\'ictor Blanco$^\dagger$ and  Justo Puerto$^\ddagger$}\medskip\\
$^\dagger$IEMath-GR, Universidad de Granada\\
$^\ddagger$IMUS, Universidad de Sevilla}

\address{IEMath-GR, Universidad de Granada, SPAIN.}
\email{vblanco@ugr.es}

\address{IMUS, Universidad de Sevilla, SPAIN.}
\email{puerto@us.es}

\date{\today}

\begin{abstract}
In this paper we propose an extension of the Uncapacitated Hub Location Problem where the potential positions of the hubs are not fixed in advance. Instead, they are allowed to belong to a region around an initial discrete set of nodes. We give a general framework in which the collection, transportation and distribution costs are based on norm-based distances and the hub-activation set-up costs depend, not only on the the location of the hub that are opened but also on the size of the region where they are placed. Two alternative mathematical programming formulations are proposed. The first one is a compact formulation while the second one involves a family of constraints of exponential size that we separate efficiently giving rise to a branch-and-cut algorithm. The results of an extensive computational experience are reported showing the advantages of each of the approaches.
\end{abstract}


\maketitle

\section{Introduction}

Hub-and-spoke networks have attracted the attention of the Locational Analysis community in the recent years since they allow to efficiently route commodities between customers in many transportation systems. In these networks, the flow between customers, rather than being sent directly user-to-user, is routed via some transshipment points, \textit{the hubs nodes}. Arcs between hubs nodes are cheaper than normal ones and thus, in this way, the overall transportation costs are reduced due to the economy of scale induced by sending a large amount of flow through the hub arcs. Hub Location Problems combine two kinds of decisions: 1) the optimal placement of the hub facilities, and 2) the best routing strategies on the induced hub-and-spoke network, i.e., the amount of flow sent through the spoke-to-hub and the hub-to-hub links. The interested reader is referred to \cite{Alumur-Kara_EJOR08} for a recent survey on hub location problems.

The literature of single-allocation hub location problems usually distinguishes between capacitated and uncapacitated problems. In the latter case, one may consider either that a given number of hubs, $p$, must be located, dealing with the $p$-hub location problem (see \cite{OKelly_EJOR87}) or that a set-up cost for each of the potential hubs is provided, and the optimal number of hubs is part of the decision process. Here, we deal with a new version of the Uncapacitated Hub Location Problem with Fixed Costs (UHLPFC). The first formulation for UHLPFC was presented by O'Kelly in 1992 \cite{OKelly_PRS92}. There, a mixed integer programming problem with quadratic (non convex) objective function was proposed for the problem. Such a difficulty was overcome by solving different $p$-hub location problems for different values of $p$. In \cite{AV98}, a different formulation was proposed based on modeling how the flow between each pair of nodes is transported between them. In \cite{Labbe-Yaman_NETW04}, the authors presented different linearizations for the quadratic terms in the objective function by using a family of $4$-index variables and projecting them out.  All the above papers deal with exact methods but also some heuristics and metaheuristic approaches have been proposed for the problem (see, for instance, \cite{AbdinnourHelm_EJOR98,TopcuogluEtal_2005}). One can also find several extensions of the UHLPFC, as the incorporation of congestion and service time \cite{AlumurEtAl_AMM18}, modular capacities \cite{Hoff_EtAl_COR17}, flow-dependent transportation costs \cite{Tanash_EtAl_COR17}, robust versions \cite{Boukani_CAM16} or the use of more sophisticated objective functions based on covering and ordered median functions~\cite{Ernst_2017,PRR11,PRR13}.

When designing a (discrete) hub-and-spoke network one usually assumes that the positions of a set of potential hubs are known, i.e., one makes the decision on the optimal location of the hubs based on the assumption that the hubs have to be located on a set of nodes whose exact placement on a given space is provided. In many cases such an assumption becomes highly restrictive, specially in those cases in which some flexibility is allowed for the location of the hubs, as in telecommunication networks or in case some imprecision affects the positions of the nodes. An   alternative to this assumption is the consideration of a continuous framework in which the hubs are allowed to be located (see \cite{OKelly_TS86}). Nevertheless, it is unrealistic in practice to assume that hubs can be located at any place. A different perspective, which is the one that we explore in this paper, comes from modelling such an issue via \textit{neighborhoods}. This framework allows us to provide a potential set of (exact) positions for the hub nodes, but in the decision process the decision-maker is allowed to place the hubs not only at their potential original geographical coordinates, but in a region around it, namely, its neighborhood. This approach is particularly useful in networks in which the decision-maker sets preferred regions where the nodes can be located, instead of an exact set of positions for them. Moreover, the continuous and the discrete cases of this problem are just particular forms of the neighborhoods version of the problem by adjusting adequately the parameters describing the neighborhoods. For instance, defining neighborhoods as singletons the uncapacitated single-allocation Hub Location Problem with variable-size Neighborhoods (UHLPN) coincides with the discrete UHLPFC while using big enough neighborhoods (and zero costs for the size of them) results in the continuous version of the problem. This approach of considering neighborhoods  is not totally new and some combinatorial optimization problems have been already analyzed under this neighborhood perspective, as the Minimum Cost Spanning Tree~\cite{Blanco_EtAl_EJOR17}, the Traveling Salesman Problem~\cite{TSPN_COR17,TSPDubinsN_IEEE15,TSPN_EJOR18}, the Shortest Path Problem or  different facility location problems~\cite{Blanco19}.  However, as far as we know, there is no previous attempt on the simultaneous determination of the location and optimal size of the neighborhoods. Observe that the neighborhoods represent locational imprecision or flexibility on the placement of the objects under analysis. Thus, fixing a pre-specified neighborhood size implies losing a degree of freedom which reduces the ability of the model to choose the right neighborhood size. This is particularly convenient in the hub-and-spoke network model under analysis. Note also that the information on the \textit{optimal} size of the neighborhood may help the decision-maker to adjust the original positions of the nodes or to restrict the sizes conveniently.

In this paper, we analyze hub location problems with some flexibility in the network design problem in a very general framework. In particular, we introduce an extension of the classical uncapacitated single-allocation hub location problem, that we call the uncapacitated single-allocation Hub Location Problem with variable-size Neighborhoods (UHLPN). In this problem we are given a set of demand points, a set of potential hubs (as coordinates in $\R^d$), an OD flow matrix between demand points, a set-up cost for opening each potential hub, a cost measure in $\R^d$ and for each potential facility, a neighborhood shape which represents some flexibility on the placement of the hub, and an additional cost based on the size of the neighborhood. The goal of UHLPN is to set up a hub-and-spoke network minimizing transportation, collection, distribution and set-up costs making decisions on:
\begin{itemize}
\item How many and  which hubs must be opened and the assignment pattern of demand points to hubs.
\item The size of the neighborhood for one each of the open hubs. The activation cost of a neighborhood may depend on the volume of the region that it defines. Smaller neighborhoods incur smaller activation costs.
\item Finding for each  demand point the location of its hub-server within the neighborhood where it must be served.
\end{itemize}

The rest of the paper is organized as follows. In Section \ref{sec:2} we introduce the elements describing the UHLPN and fix the notation for the rest of the sections. Section \ref{sec:3} provides a mathematical programming formulation for the problem as well as two different reformulations. While the first one is based on linearizing the bilinear and trilinear terms in the original formulation, the second one consists of replacing some of the non linear constraints by a family of linear constraints of exponential size. Based on the latter, in Section \ref{sec:4} we derive a brach-and-cut approach to solve the HLPN. In Section \ref{sec:5} we report the results of our extensive computational experience based on the usual hub location datasets. Finally, in Section \ref{sec:6} we draw some conclusions and further research on the topic.

\section{Uncapacitated Single-Allocation Hub Location Problem with Neighborhoods}\label{sec:2}

In this section we formally introduce the UHLPN and fix the notation for the rest of the paper. We describe here the input data of the problem, the feasible solutions of the UHLPN and the objective function of the problem that allows us to evaluate the feasible alternatives.

\subsection{Input Data}

We are given an undirected connected graph, $G=(N,E)$, with nodes $N=\{1, \ldots, n\}$. The graph is assumed to be embedded in $\R^d$, i.e., each node $i\in N$ is positioned at point $a_i\in \R^d$. For each pair of nodes $i, j\in N$ there is a known demand, $w_{ij}\ge 0$, which represents the amount of flow that must be sent (\emph{routed}) from origin $i$ to destination $j$. Each origin/destination pair of nodes with positive demand will be referred to as an OD pair. We denote by $W=(w_{ij})_{i,j\in N}$ the OD flow matrix, by $O_i=\dsum_{j\in N:\atop \{i,j\}\in E} w_{ij}$ the overall flow with origin at $i\in N$, and by $D_{j}=\dsum_{i:\atop \{i,j\} \in E} w_{ij}$  the total flow with destination at $j\in N$.

 \subsection{Feasible Actions}

 Like in other hub location models, in the UHLPN a set of hubs must be selected, among the nodes in $N$, to be used as potential intermediate points when routing the flows associated with OD pairs.  In particular, the flow associated with a given OD pair $(i, j)$ is assumed to be \emph{physically} routed via a \emph{feasible path} of the form $(a_i, x_k, x_m, a_j)$ where $x_k$, and $x_m$ are open hubs associated with locations $k, m\in N$, respectively, and it is possible that $k=m$ (so $x_k=x_m$).  When nodes $i$ and $j$ are both selected to host hubs, these paths reduce to $(x_i, x_i, x_j, x_j)$ and consist of one single arc. Otherwise, if $i$ is not a hub then $x_k\ne a_i$; similarly, $x_m\ne a_j$ if $j$ is not a hub.

We assume \emph{single allocation} of nodes to open hubs. That is, for each node $i\in N$ all the flows \emph{leaving} and \emph{entering} $i$ must be routed via a single (unique) hub node, which represents the access and exit point of $i$ to and from the distribution system.

Contrary to classical hub location models, in the UHLPN there is not a discrete set of points where hubs may be located at, even if the problem is stated on a graph with a discrete set of nodes. In particular, when node $k\in N$ is selected to host an open hub, the actual position of the hub is not known in advance and must be set within a neighbourhood of the  selected node $k$. For this,  associated with each potential hub $k\in N$ we are given a \emph{basic neighbourhood}, $S_k \subseteq \R^d$, which is assumed to be a second order cone (SOC) representable set containing the origin. The neighborhood set for $a_k$, then, is assumed to be in the form
$$
\mathcal{N}_k (r) = a_k+ \{r \cdot z: z \in S_k\},
$$
for $r\geq 0$, i.e., the $a_k$-translation of the $r$-dilation of the set $S_k$. The choice of $r$, together with the specific location of the hub, $x_k$, is part of the decision making process.

Thus, we assume that feasible paths are in the form $(a_i, x_k, x_m, a_j)$ where $a_i$ and $a_j$ correspond to OD pairs and $x_k$ and $x_m$ belong to the dilated neighborhoods of nodes $k$ and $m$ in the graph $G$.

In summary, for determining a feasible UHLPN solution the following decisions must be made: (1) the nodes that host the open hubs into their neighborhoods; (2) the dilation factor applied to the basic neighbourhood of each selected node, as well as the actual position of its associated hub within its dilated neighbourhood; and, (3) the (single) allocation of non-hub nodes (\emph{spokes}) to open hubs.

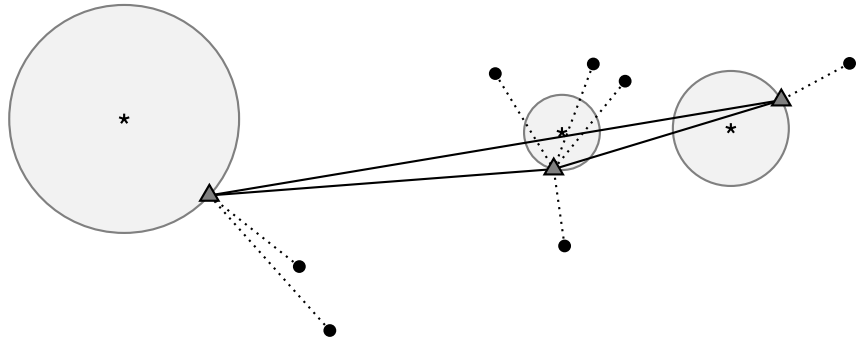
\begin{figure}[h]
\begin{center}
\begin{tikzpicture}[scale=2]

\begin{axis}[
hide x axis,
hide y axis,
xmin=-111,
ymax=46,
x=4,
y=4,
]
\draw[ draw=gray,fill=gray, fill opacity=0.1] (axis cs:-76.6121893,39.2903848) circle (2.71178021138157);
\draw[ draw=gray,fill=gray, fill opacity=0.1](axis cs:-84.5120196,39.1031182) circle (1.77284477731651);
\draw[ draw=gray,fill=gray, fill opacity=0.1] (axis cs:-104.990251,39.7392358) circle[radius=5.37603398242093];

\addplot [black, densely dotted]
table {%
-84.3879824 33.7489954
-84.8975432152862 37.3726992863417
};
\addplot [black, densely dotted]
table {%
-71.0588801 42.3600825
-74.2525317815513 40.6267103786935
};
\addplot [black, densely dotted]
table {%
-87.6297982 41.8781136
-84.8975432152862 37.3726992863417
};
\addplot [black, densely dotted]
table {%
-81.556235 41.5200518
-84.8975432152862 37.3726992863417
};
\addplot [black, densely dotted]
table {%
-96.7969879 32.7766642
-101.005023872329 36.1309528376735
};
\addplot [black, densely dotted]
table {%
-83.0457538 42.331427
-84.8975432152862 37.3726992863417
};
\addplot [black, densely dotted]
table {%
-95.3698028 29.7604267
-101.005023872329 36.1309528376735
};

\addplot [mark=triangle*,  fill=gray,mark size=2, mark options={solid}, only marks]
table {%
-74.2525317815513 40.6267103786935
};
\addplot [mark=triangle*,  fill=gray,mark size=2, mark options={solid}, only marks]
table {%
-101.005023872329 36.1309528376735
};
\addplot [mark=triangle*, fill=gray,mark size=2, mark options={solid}, only marks]
table {%
-84.8975432152862 37.3726992863417
};
\addplot[black]
table {%
-84.8975432152862 37.3726992863417
-74.2525317815513 40.6267103786935
};
\addplot[black]
table {%
-101.005023872329 36.1309528376735
-74.2525317815513 40.6267103786935
};

\addplot[black] 
table {%
-101.005023872329 36.1309528376735
-84.8975432152862 37.3726992863417
};
\addplot [black, mark=*, mark size=1, mark options={solid}, only marks]
table {%
-84.3879824 33.7489954
};
\addplot [black, mark=star, mark size=1, mark options={solid}, only marks]
table {%
-76.6121893 39.2903848
};
\addplot [black, mark=*, mark size=1,mark options={solid}, only marks]
table {%
-71.0588801 42.3600825
};
\addplot [black, mark=*, mark size=1,mark options={solid}, only marks]
table {%
-87.6297982 41.8781136
};
\addplot [black, mark=star, mark size=1, mark options={solid}, only marks]
table {%
-84.5120196 39.1031182
};
\addplot [black, mark=*, mark size=1,mark options={solid}, only marks]
table {%
-81.556235 41.5200518
};
\addplot [black, mark=*, mark size=1,mark options={solid}, only marks]
table {%
-96.7969879 32.7766642
};
\addplot [black, mark=star, mark size=1, mark options={solid}, only marks]
table {%
-104.990251 39.7392358
};
\addplot [black, mark=*, mark size=1, mark options={solid}, only marks]
table {%
-83.0457538 42.331427
};
\addplot [black, mark=*, mark size=1,mark options={solid}, only marks]
table {%
-95.3698028 29.7604267
};

\end{axis}

\end{tikzpicture}
\end{center}
\caption{Feasible instance for the UHLPN.\label{fig:0}}
\end{figure}

A possible choice for modeling different SOC-representable neighborhoods, is by means of unit balls of norms, i.e., when $S_k = \{z \in \R^d: \|z\| \leq 1\}$ for some norm $\|\cdot\|$ in $\R^d$, which belong to the class of $\ell_p$-norm or polyhedral norms~\cite{BPE14}. For instance, if the basic neighbours are Euclidean unit balls, dilations by $r$  consist of Euclidean balls of radius $r$ and center $a_k$, while for polyhedral norms, the resulting neighbourhoods are dilations of symmetric polytopes. One may also consider more sophisticated shapes for the basic neighborhoods, as for instance polyellipsoids~\cite{BP19Polyellipses}. We assume that for every potential hub location $k\in N$, there is an upper bound, $R_k \in \R_+$, on the maximum size of the dilation factor of its associated neighbourhood.  \\

In Figure \ref{fig:0} we show a feasible instance of our problem. Black dots are the spoke nodes. Gray disks ($\ell_2$-balls on the plane) are the dilated neighborhoods centered at the initial position of the active hub nodes (asterisks). The position on the neighborhoods of the hubs are the triangles in the picture. Finally, dotted lines represent allocation between spoke and hub nodes and the solid lines form the hub backbone network.
\subsection{Costs}

The cost of a feasible solution is the sum of the set-up costs of the activated hubs plus the transportation costs for routing the commodities, through the allocated hubs.
\begin{description}
\item[Set-up costs:] Activating a hub associated with node $k\in N$ brings a set-up cost $F_k(r)$, which depends on $k$ and the dilation factor, $r$, that is applied to the basic neighborhood. In particular, we consider set-up costs of the form:
$$
F_k(r) = f_k + g_k(r),
$$
where the term $f_k$ represents a fixed cost for establishing a hub at node $k$, whereas $g_k: \R_+ \rightarrow \R$ is a second order cone representable function of the dilation factor, verifying that $g_k(r) \geq 0$, for all $r\geq 0$, and $g_k(0)=0$  (null dilations have zero cost). Some possible choices of $g_k$ are those in which the cost of installing a hub with $r$-dilation increases linearly with $r$, being then $g_k(r) = \Lambda_k r$ for a given coefficient $\Lambda_k\geq 0$; or a polynomial function $g_k(r) = \Lambda_k r^d$ which may represent the per unit (area/volume) installation cost of a neighbourhood associated with node $k$.

\item[Routing Costs:]  Sending flow though the path $(a_i, x_k, x_m, a_j)$ incurs a unit routing cost $\d^{C}(a_i, x_k)+\d^H(x_k, x_m) +\d^{D}(x_m,a_j)$, where  $\d^{C}, \d^{H}, \d^D: \R^d\times \R^d \rightarrow \R_+$ are given cost/distance functions for the collection, transportation and distribution of unitary  flows, respectively. To reflect economies of scale between hub nodes, we assume that $\d^H(x,y) < \d^C(x,y)$ and $\d^H(x,y)<\d^D(x,y)$, for all $x, y \in \R^d$, i.e., the cost of sending a unit of flow through two hub node  positions (in the hub backbone network) is smaller than the one of using spoke-to-hub links between the same positions. Moreover, we assume that in all cases, $\d^C(x,x)= \d^H(x,x)= \d^D(x,x)=0$ for all $x\in \R^d$.  A very interesting case which generalizes the usual practices for economies of scale in hub location problems is to consider that $\d^C = \d^D$ and $\d^H = \alpha \d^C$, for some $\alpha \in (0,1)$ for some distance measure $\d^C$ in $\R^d$. Also, using the relationship between different $\ell_p$-norm based distances ($\d_{\ell_{p_1}} < d_{\ell_{p_2}}$ whenever $p_1>p_2\geq 1$), one could also choose $\d^C=\d^D = \d_{\ell_{p_1}}$ and $\d^H=\d_{\ell_{p_2}}$ for any $p_1, p_2 \geq 1$ (here $\d_{\ell_p}$ stands for the distance induced by the $\ell_p$ norm)  In Figure \ref{routingcosts} we illustrate the different routing costs considered in the UHLPN, in a case in which $\d^C=\d^D=\d_{\ell_1}$ and $\d^H=\d_{\ell_2}$. The costs of routing a unit flow from node $i$ to node $j$ are highlighted with dotted lines (dashed lines correspond with collection and distribution connections while solid lines are hub links).
\end{description}
\begin{figure}[h]
\begin{center}

\begin{tikzpicture}[scale=0.75]

\begin{axis}[
hide x axis,
hide y axis,
x=0.3cm,
y=0.3cm,
]

\addplot [very thin, black, dashed]
table {%
20.355966023 16.167127237
23.572548971 16.167127237
};
\addplot [very thin, black, dashed]
table {%
23.572548971 16.167127237
23.572548971 37.6328299723607
};
\addplot [very thick, black, dotted,<-]
table {%
39.98859202 19.773197847
38.2188846318043 19.773197847
};
\addplot [very thick, black, dotted,<-]
table {%
38.2188846318043 19.773197847
38.2188846318043 32.1551769960305
};
\addplot [very thin, black, dashed]
table {%
23.572548971 31.529184022
23.572548971 31.529184022
};
\addplot [very thin, black, dashed]
table {%
23.572548971 31.529184022
23.572548971 37.6328299723607
};
\addplot [very thin, black, dashed]
table {%
40.867253869 38.565884488
38.2188846318043 38.565884488
};
\addplot [very thin, black, dashed]
table {%
38.2188846318043 38.565884488
38.2188846318043 32.1551769960305
};
\addplot [very thin, black, dashed]
table {%
27.520840478 46.921515986
32.6905287105151 46.921515986
};
\addplot [very thin, black, dashed]
table {%
32.6905287105151 46.921515986
32.6905287105151 47.8833010221345
};
\addplot [very thin, black, dashed]
table {%
36.067877874 44.894490748
32.6905287105151 44.894490748
};
\addplot [very thin, black, dashed]
table {%
32.6905287105151 44.894490748
32.6905287105151 47.8833010221345
};
\addplot [very thick, black, dotted, ->]
table {%
19.345710923 51.97204094
32.6905287105151 51.97204094
};
\addplot [very thick, black, dotted, ->]
table {%
32.6905287105151 51.97204094
32.6905287105151 47.8833010221345
};
\addplot [very thin]
table {%
23.572548971 37.6328299723607
38.2188846318043 32.1551769960305
};
\addplot [very thin, fill=gray, mark=triangle*, mark size=0.5,  mark options={solid}, only marks]
table {%
38.2188846318043 32.1551769960305
};
\addplot [very thin, fill=gray, mark=triangle*, mark size=0.5,  mark options={solid}, only marks]
table {%
32.6905287105151 47.8833010221345
};
\addplot [very thick, dotted, ->]
table {%
32.6905287105151 47.8833010221345
38.2188846318043 32.1551769960305
};
\addplot [very thin, blue, mark=*, mark size=0.5,  mark options={solid}, only marks]
table {%
23.572548971 37.6328299723607
};
\addplot [very thin]
table {%
32.6905287105151 47.8833010221345
23.572548971 37.6328299723607
};
\addplot [very thin, black, mark=*, mark size=1,  mark options={solid}, only marks]
table {%
20.355966023 16.167127237
};
\addplot [very thin, black, mark=*, mark size=1,  mark options={solid}, only marks]
table {%
39.98859202 19.773197847
};
\addplot [very thin, black, mark=*, mark size=1,  mark options={solid}, only marks]
table {%
23.572548971 31.529184022
};
\addplot [very thin, black, mark=asterisk, mark size=1, mark options={solid}, only marks]
table {%
37.30436959 32.079249435
};
\addplot [very thin, black, mark=asterisk, mark size=1, mark options={solid}, only marks]
table {%
24.512513179 38.835826089
};
\addplot [very thin, black, mark=*, mark size=1,  mark options={solid}, only marks]
table {%
40.867253869 38.565884488
};
\addplot [very thin, black, mark=*, mark size=1,  mark options={solid}, only marks]
table {%
27.520840478 46.921515986
};
\addplot [very thin, black, mark=*, mark size=1,  mark options={solid}, only marks]
table {%
36.067877874 44.894490748
};
\addplot [very thin, black, mark=*, mark size=1,  mark options={solid}, only marks]
table {%
19.345710923 51.97204094
};
\addplot [very thin, black, mark=asterisk, mark size=1, mark options={solid}, only marks]
table {%
33.635749233 49.663479404
};

\draw[draw=gray,fill=gray,opacity=0.45] (axis cs: 37.30436959,32.079249435) circle [radius=0.3*0.917657250849806cm];
\draw[draw=gray,fill=gray,opacity=0.45] (axis cs: 24.512513179,38.835826089) circle  [radius=0.3*1.52667362730612cm];
\draw[draw=gray,fill=gray,opacity=0.45] (axis cs: 33.635749233,49.663479404) circle  [radius=0.3*2.01555724463284cm];

\node[left] at (axis cs: 19.345710923, 51.97204094) {$a_i$};
\node[left] at (axis cs: 32.6905287105151, 47.8833010221345) {$x_k$};
\node[right] at (axis cs: 38.2188846318043, 32.1551769960305) {$x_m$};
\node[right] at (axis cs: 39.98859202, 19.773197847) {$a_j$};

\node[left] at (axis cs: 38,25) {\small $\d^D(x_m, a_j)$};
\node[right] at (axis cs: 36,41) {\small $\d^H(x_k, x_m)$};
\node[below] at (axis cs: 28,51) {\small $\d^C(a_i, x_k)$};
\end{axis}

\end{tikzpicture}
\end{center}
\caption{Routing costs in the UHLPN.\label{routingcosts}}
\end{figure}
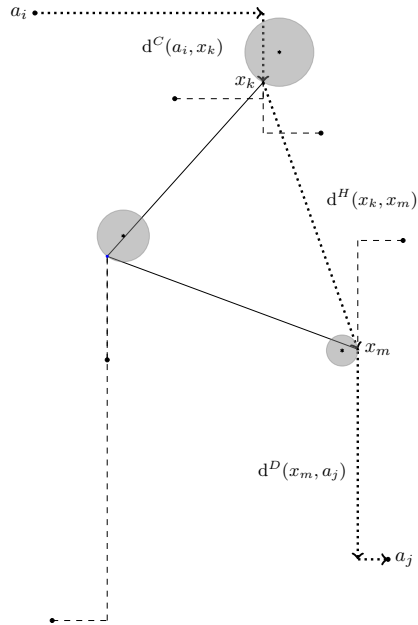
The UHLPN is to find a feasible solution of minimum total cost. Thus, the UHLPN generalizes classical hub location problems with facilities set-up costs where neighbourhoods are not allowed (upper bound for the dilations $R=0$), and hubs must be necessarily located at their associated points $a_i$, $i\in N$. These problems, can  also be obtained as particular cases of the UHLPN by setting to an arbitrarily large value all the coefficients $\Lambda_k$, $k\in K$, thus enforcing all the dilations to be zero.

Figure \ref{fig0} shows the solution of a ten node AP instance~\cite{Ernst96} in two different situations: without  neighbourhoods (left), and with $\ell_2$-norm neighborhoods.
Dashed lines represent spoke allocations to open hubs and solid lines inter-hub connections. Observe that the optimal design of the network changes if some flexibility is allowed to the hubs for their location.
\begin{figure}[h]\begin{center}
\begin{tikzpicture}

\begin{axis}[
hide x axis,
hide y axis,
axis equal,
]
\draw[fill=gray,opacity=0.3] (axis cs:37.30436959,32.079249435) circle (0.972006454544975);
\draw[fill=gray,opacity=0.3] (axis cs:33.635749233,49.663479404) circle (1.69958504074655);

\addplot [semithick, black, dashed]
table {%
20.355966023 16.167127237
36.3589276286313 31.853559367394
};
\addplot [semithick, black, dashed]
table {%
39.98859202 19.773197847
36.3589276286313 31.853559367394
};
\addplot [semithick, black, dashed]
table {%
23.572548971 31.529184022
36.3589276286313 31.853559367394
};
\addplot [semithick, black, dashed]
table {%
24.512513179 38.835826089
32.4807250503454 48.4166801259713
};
\addplot [semithick, black, dashed]
table {%
40.867253869 38.565884488
36.3589276286313 31.853559367394
};
\addplot [semithick, black, dashed]
table {%
27.520840478 46.921515986
32.4807250503454 48.4166801259713
};
\addplot [semithick, black, dashed]
table {%
36.067877874 44.894490748
32.4807250503454 48.4166801259713
};
\addplot [semithick, black, dashed]
table {%
19.345710923 51.97204094
32.4807250503454 48.4166801259713
};
\addplot [semithick, fill=gray, mark=triangle*, mark size=3, mark options={solid}, only marks]
table {%
36.3589276286313 31.853559367394
};
\addplot [semithick, fill=gray, mark=triangle*, mark size=3, mark options={solid}, only marks]
table {%
32.4807250503454 48.4166801259713
};
\addplot [semithick]
table {%
32.4807250503454 48.4166801259713
36.3589276286313 31.853559367394
};
\addplot [semithick, black, mark=*, mark size=2, mark options={solid}, only marks]
table {%
20.355966023 16.167127237
};
\addplot [semithick, black, mark=*, mark size=2, mark options={solid}, only marks]
table {%
39.98859202 19.773197847
};
\addplot [semithick, black, mark=*, mark size=2, mark options={solid}, only marks]
table {%
23.572548971 31.529184022
};

\addplot [semithick, black, mark=*, mark size=2, mark options={solid}, only marks]
table {%
24.512513179 38.835826089
};
\addplot [semithick, black, mark=*, mark size=2, mark options={solid}, only marks]
table {%
40.867253869 38.565884488
};
\addplot [semithick, black, mark=*, mark size=2, mark options={solid}, only marks]
table {%
27.520840478 46.921515986
};
\addplot [semithick, black, mark=*, mark size=2, mark options={solid}, only marks]
table {%
36.067877874 44.894490748
};
\addplot [semithick, black, mark=*, mark size=2, mark options={solid}, only marks]
table {%
19.345710923 51.97204094
};
\end{axis}

\end{tikzpicture}~\hspace*{1.5cm}~
\begin{tikzpicture}

\begin{axis}[
hide x axis,
hide y axis,
axis equal,
]
\draw[draw=red,fill=red,opacity=0.3] (axis cs:37.30436959,32.079249435) circle (0);
\draw[draw=red,fill=red,opacity=0.3] (axis cs:24.512513179,38.835826089) circle (0);
\draw[draw=red,fill=red,opacity=0.3] (axis cs:33.635749233,49.663479404) circle (0);

\addplot [semithick, black, dashed]
table {%
20.355966023 16.167127237
24.512513179 38.835826089
};
\addplot [semithick, black, dashed]
table {%
39.98859202 19.773197847
37.30436959 32.079249435
};
\addplot [semithick, black, dashed]
table {%
23.572548971 31.529184022
24.512513179 38.835826089
};
\addplot [semithick, black, dashed]
table {%
40.867253869 38.565884488
37.30436959 32.079249435
};
\addplot [semithick, black, dashed]
table {%
27.520840478 46.921515986
33.635749233 49.663479404
};
\addplot [semithick, black, dashed]
table {%
36.067877874 44.894490748
33.635749233 49.663479404
};
\addplot [semithick, black, dashed]
table {%
19.345710923 51.97204094
33.635749233 49.663479404
};
\addplot [semithick, mark=triangle*, fill=gray,mark size=3, mark options={solid}, only marks]
table {%
37.30436959 32.079249435
};
\addplot [semithick, mark=triangle*, fill=gray,mark size=3, mark options={solid}, only marks]
table {%
24.512513179 38.835826089
};
\addplot [semithick]
table {%
24.512513179 38.835826089
37.30436959 32.079249435
};
\addplot [semithick,  mark=triangle*, fill=gray,mark size=3, mark options={solid}, only marks]
table {%
37.30436959 32.079249435
};
\addplot [semithick,mark=triangle*, fill=gray, mark size=3, mark options={solid}, only marks]
table {%
33.635749233 49.663479404
};
\addplot [semithick]
table {%
33.635749233 49.663479404
37.30436959 32.079249435
};
\addplot [semithick]
table {%
33.635749233 49.663479404
24.512513179 38.835826089
};
\addplot [semithick, black, mark=*, mark size=2, mark options={solid}, only marks]
table {%
20.355966023 16.167127237
};
\addplot [semithick, black, mark=*, mark size=2, mark options={solid}, only marks]
table {%
39.98859202 19.773197847
};
\addplot [semithick, black, mark=*, mark size=2, mark options={solid}, only marks]
table {%
23.572548971 31.529184022
};
\addplot [semithick, black, mark=*, mark size=2, mark options={solid}, only marks]
table {%
40.867253869 38.565884488
};
\addplot [semithick, black, mark=*, mark size=2, mark options={solid}, only marks]
table {%
27.520840478 46.921515986
};
\addplot [semithick, black, mark=*, mark size=2, mark options={solid}, only marks]
table {%
36.067877874 44.894490748
};
\addplot [semithick, black, mark=*, mark size=2, mark options={solid}, only marks]
table {%
19.345710923 51.97204094
};
\end{axis}

\end{tikzpicture}
\end{center}
\caption{Solutions of an instance of the AP dataset.\label{fig0}}
\end{figure}
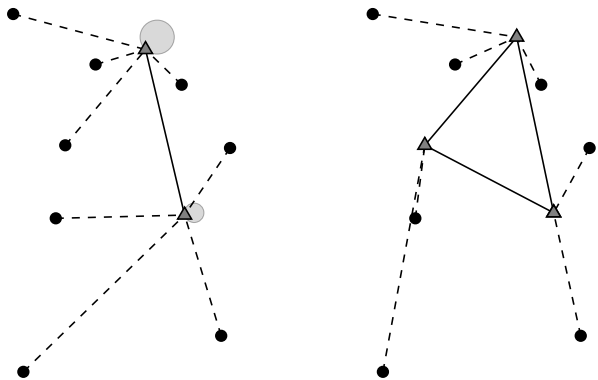

\section{Mathematical Programming Formulations}\label{sec:3}

In this section we develop alternative mathematical programming formulations for the UHLPN. The first one is a Mixed-Integer Non-Linear Programming (MINLP) formulation, which includes bilinear and trilinear terms in the objective function. These terms come from the products of  binary variables with  binary and/or continuous variables. In addition, it also includes non-linear terms in the constraints  modeling  the membership of hubs to their dilated neighbourhoods. Then we present two reformulations. The first one introduces additional decision variables in order to linearize the non-linear terms of the original formulation whereas the second one is obtained from the aggregation of some of the decision variables introduced for the first reformulation.

\subsection*{Decision variables}

All our formulations use the following common sets of decision variables:

\begin{itemize}
\item \textit{Location-allocation variables.} They determine the nodes selected for locating hubs, as well as the allocation of nodes to open hubs. For every edge $\{i, k\} \in E$:
$$
z_{ik} = \left\{\begin{array}{cl}
1 & \mbox{if node $i$ is allocated to hub $k$, for $i\neq k$},\\
0 & \mbox{otherwise,}
\end{array}\right.
$$
and
$$
y_{k} = \left\{\begin{array}{cl}
1 & \mbox{if node $k$ is an open hub},\\
0 & \mbox{otherwise}
\end{array}\right..
$$
\item \textit{Radii of the neighbourhoods of the selected hubs.} We consider variables $r_i\in \mathbb{R}_+$, for $i=1,\ldots,n$. Each $r_i$  determines the dilation factor applied to the neighbourhood of the potential hub node $i$  if it were activated (and zero otherwise).\\

\item \textit{\textit{Positions of hubs within the dilated neighbourhoods associated with the open hubs.} The position of hub $k\in N$ is given by $x_{k}  \in \R^d$.}

\end{itemize}

Using the above decision variables the routing costs can be expressed as:
$$
f_{\rm rout} (z,x) = \dsum_{\{i,k\}\in E} (O_i\d^C(a_i, x_{k})  z_{ik}+ D_i\d^D(a_i, x_{k}))  z_{ik} + \dsum_{k, m\in N} \dsum_{i,j\in N:\atop\{i,k\}, \{j,m\} \in E} w_{ij} \d^H(x_{k},x_{m}) z_{ik} z_{jm}.
$$
The term $\dsum_{\{i,k\}\in E} O_i\d^C(a_i, x_{k})  z_{ik}$ accounts for the collection costs induced for sending flows from non-hub nodes (spokes) to their allocated hubs. The term $\dsum_{\{i,k\}\in E} D_i\d^D(a_i, x_{k}) z_{ik}$ accounts for the distribution costs due to sending flows from hubs to the (final)  spokes assigned to them. The last term determines the inter-hub transportation costs: each flow $w_{ij}$ associated with the OD pair $(i, j)$ incurs an inter-hub cost, which depends on the hubs to which $i$ and $j$ are allocated. Note that the inter-hub transportation cost of an OD pair $(i, j)$ in which $i$ and $j$ are allocated to the same hub is zero.

Furthermore, the set-up costs for the activated hubs, taking into account the radii of their associated dilated neighborhoods is given by
$$
f_{\rm set-up} (y,r) =  \dsum_{k\in N} F_k(r_k) y_k.
$$
Thus, the overall objective function of the problem is:
$$
f(z,y,r,x) = f_{\rm rout} (z,x)  +f_{\rm set-up} (y,r)
$$
A mathematical programming formulation for the problem, based on the one proposed by O'Kelly~\cite{OKelly_EJOR87} is:
\begin{align}
\min & \;\; f(z,r,x)\nonumber \\
\mbox{s.t. } & \dsum_{k \in N: \atop \{i,k\} \in E} z_{ik} =1, &&\forall i \in N,\label{uf0:c1}\\
& z_{ik} \leq y_{k}, &&\forall \{i, k\} \in E,\label{uf0:c2}\\
&  x_{k} \in \mathcal{N}_k(r_k), &&\forall k \in N,\label{uf0:c3}\\
& z_{ik}, y_k \in \{0,1\}, &&\forall \{i, k\} \in E,\nonumber \\
& 0 \leq r_k \leq R_k, &&\forall k\in N. \nonumber
\end{align}
Constraints \eqref{uf0:c1} ensure that each spoke is assigned to a single hub node, while constraints \eqref{uf0:c2} state that allocations are only permitted to open hubs. Finally, constraints \eqref{uf0:c3} establish that the access points of assignments to hubs belong to their dilated neighbourhoods. Observe that if node $k\in N$ is not selected as a hub node, its dilation factor $r_k$ becomes zero since otherwise it would incur in a positive set-up cost.

The above model is a Mixed-Integer Non-Linear Programming formulation, not only because the products of $z$-variables in the objective function ($f_{\rm rout}$) but also because the terms $g_k(r_k) y_{k}$ in $f_{\rm set-up}$, the distances between nodes and the membership to the neighborhoods (Constraint \ref{uf0:c3}).

Some observations are in order concerning the $r$-variables and their representation in the objective function. First, observe that the non-linear term  $g_k(r_k) y_{k}$ appears in the objective function to represent the cost induced by the actual size of the dilated neighbourhood associated with hub $k$. However,  in case $g_k$ is a non decreasing function, one can replace the term $g_k(r_k) y_{k}$ by just $g_k(r_k)$. This can be done because the minimization criterion and the non-negativity of $g_k$ guarantee that in case hub $k$ is not open, $g_k(r_k)$ will take the smallest possible value, i.e., zero. Otherwise, $y_{k}=1$, so  $g_k(r_k) y_{k}=g_k(r_k)$. Thus, for the sake of simplicity, we assume from now on that $g_k$ is a non decreasing function, although one could apply similar strategies for general choices of $g_k$.

As mentioned above, different shapes are possible for the functions $g_k$. In case $g_k$ is a linear function, the overall function $f_{\rm set-up}$ is linear. If $g_k(r) = \Lambda_k r^d$, the term $r_k^d$ in the objective function  can be adequately rewritten as a linear function by adding a new auxiliary variable and representing it as a set of second order cone constraints. In particular, denoting by $\gamma_k$ the $d$-th power of $r_k$, i.e., $\gamma_k=r_k^d$, we get that $f_{set-up}(y,r) =  \dsum_{k\in N} f_k y_k + \dsum_{k\in N} \Lambda_k\gamma_k$. Then, constraints in the form $\gamma_k \geq r_k^d$ allows us to represent such a term in the mathematical programming model as a small number of rotated second order cone constraints. The next result reformulates the $d$-th powers in the objective function by means of a set of $O(n d)$  new variables and second order cone constraint.
\begin{lem}
Let $\alpha = (\alpha_0, \ldots, \alpha_{q-1}) \in \{0,1\}^q$ be the coefficients of the binary decomposition of $d-1$, i.e., $d-1 = \alpha_0 2^0 + \alpha_1 2^1 + \cdots, + \alpha_{q-1} 2^{q-1}$ with $q\in \Z_+$ such that $2^{q-1} \leq d < 2^q$. Then, for each $k\in N$, constraint $\gamma_k \geq r_k^d$ can be equivalently represented as the following set of $q$ second order cone/linear  constraints:
\begin{align*}
r_k^2 &\leq \omega_{q-1} \cdot r_k^{(1-\alpha_{q-1})},\\
\omega_{i+1}^2 & \leq \omega_i \cdot r_k^{(1-\alpha_i)}, \text{ for } i=1, \ldots, q-2,\\
\omega_1^2 &\leq \gamma_k \cdot r^{(1-\alpha_0)},
\end{align*}
where $\omega_1, \ldots, \omega_{q-1}$ are nonnegative auxiliary variables.
\end{lem}
\begin{proof}
The proof follows by using Lemma 1 in \cite{BPE14} to represent constraint $r^{2^q} \leq \gamma r^{2^q - d}$ (which is equivalent to the one to be represented). First, the binary representations of $1$, $d-1$ and $2^q -1$ are computed, i.e.,
\begin{align*}
1 &= 1 \times 2^0 + 0 \times 2^1 + \cdots + 0 \times 2^{q-1},\\
d-1 &= \alpha_0 \times 2^0 + \alpha_1 \times 2^1 + \cdots + \alpha_{q-1} \times 2^{q-1},\\
2^{q}-d &= \beta_0 \times 2^0 + \beta_1 \times 2^1 + \cdots + \beta_{q-1} \times 2^{q-1}.
\end{align*}
Observe that the above decompositions allow to represent $2^q$ not only as $1 + \dsum_{i=1}^{q-1} 2^i$ but also as:
$$
2^q = (1+\alpha_0+\beta_0)+ (\alpha_1+\beta_1)\times 2 + (\alpha_2+\beta_2)\times 2^2 + \cdots + (\alpha_{q-1}+\beta_{q-1}) \times 2^{q-1}
$$
Following similar arguments to those explicitly detailed in \cite{BPE14}, we get that $\alpha_i+\beta_i=1$ for all $i=0, \ldots, q-1$. Thus, one can represent constraint $r^{2^q} \leq \gamma r^{2^q - d}$ by concatenating the $q$ rotated second order cone constraints stated in the result and then:
\begin{eqnarray*}
\begin{split}
r^{2^q} &\leq w_{q-1}^{2^{q-1}} r^{2^{q-1} (1-\alpha_{q-1})} \leq w_{q-2}^{2^{q-2}} r^{2^{q-1}(1-\alpha_{q-1}) + 2^{q-2}(1-\alpha_{q-2})} \\
&\leq \cdots \leq w_1^{2} r^{2^{q-1}(1-\alpha_{q-1}) + 2^{q-2}(1-\alpha_{q-2}) + \cdots + 2(1-\alpha_1)}\\
& \leq \gamma  r^{2^{q-1}(1-\alpha_{q-1}) + 2^{q-2}(1-\alpha_{q-2}) + \cdots + 2^1(1-\alpha_1) + 2^0(1-\alpha_0)}\\
& = \gamma  r^{2^q - 1} r^{- \sum_{i=0}^{q-1} \alpha_i 2^i} = \gamma r^{2^q-1} r^{-(d-1)}= \gamma r^{2^q - d}
\end{split}
\end{eqnarray*}

Thus, it implies that $r^d \leq \gamma$.

\end{proof}

Using the above result one can easily represent volumes of the neighborhoods. For instance,  for $d=3$ and $d=4$, the terms $r_k^3$ or $r^4_k$ can be represented by means of the following set of SOC-inequalities and auxiliary variables:
\begin{center}
\begin{multicols}{2}
$\left\{\begin{array}{rl}
\gamma_k r_k &\geq \omega_1^2\\
\omega_1 &\geq r_k^2
\end{array}\right.$

\columnbreak

$\left\{\begin{array}{rl}
\gamma_k &\geq \omega_1^2\\
\omega_1 &\geq \omega_2^2\\
\omega_2 r_k &\geq r_k^2
\end{array}\right..$
\end{multicols}
\end{center}

The following result states a geometrical property of the optimal solutions of the UHLPN.
\begin{prop}
Optimal locations of the hubs ($x$-variables) must belong to the boundary of the selected neighborhoods.
\end{prop}
\begin{proof}
Let $(x^*, r^*, z^*)$ be the optimal solution for UHLPN with objective value $f^*$. Assume that there is some $k\in N$ such that $z^*_{kk}=1$ and $x_k^* \in {\rm int}(\mathcal{N}_k(r^*_k))$. Thus, there exists $0<\varepsilon_k<r^*_k$ such that $x_k^* \in \mathcal{N}(r_k^*-\varepsilon)$. Let us define $r^\prime = (r_1^*, \ldots, r^*_{k-1}, r^*_k-\varepsilon_k, r^*_{k+1}, \ldots, r_n^*)$. Clearly, $(x^*, r^\prime, z^*)$ is still feasible and its objective function is smaller than $f^*$ since it only affects the term $g_k$ and $g_k(r^*_k) \geq g_k(r^\prime_k)$, contradicting the optimality of $(x^*, r^*, z^*)$.
\end{proof}

\subsection{First Reformulation}\label{f1}

Next we present a reformulation of the UHLPN described in the section above based on the linearization of the non-linear terms of the objective function.\\
Indeed, the bilinear and trilinear terms in the objective function can  be linearized by introducing auxiliary decision variables. In particular, for all $\{i,k\}, \{j,m\}, \{k,m\} \in E$ let:
$$
\eta^C_{ik} = \d^C(a_i, x_{k}) z_{ik}, \;\; \eta^D_{ik} =\d^D(a_i, x_{k})  z_{ik} \mbox{ and } \nu_{ikjm} = \d^H(x_{k},x_{m}) z_{ik} z_{jm}.
$$

Using the new set of decision variables, the UHLPN reduces to:

\begin{align}
\min \, & \dsum_{i,k\in N}  (O_i \eta^C_{ik} + D_i \eta^D_{ik}) +\dsum_{i,j\in N} w_{ij}  \dsum_{\{i,k\},\{j,m\},\{k,m\} \in E} \nu_{ikjm} &+ \dsum_{k\in N} f_k y_k + \dsum_{k\in N} g_k(r_k)\nonumber\\
\mbox{s.t. } & \eqref{uf0:c1}, \eqref{uf0:c2}, \eqref{uf0:c3},\nonumber\\
& \eta^C_{ik} \geq \d^C(a_i, x_{k})- \hat{D}^C_{ik} (1-z_{ik}), &\forall \{i, k\} \in E,\label{uf1:c1}\\
& \eta^D_{ik} \geq \d^D(x_k, a_i) - \hat{D}^D_{ik} (1-z_{ik}), &\forall \{i, k\} \in E,\label{uf1:c2}\\
& \nu_{ikjm} \geq   \d^H(x_{k},x_{m}) - \hat{D}^H_{km}(2- z_{ik}-z_{jm}), &\forall \{i, k\}, \{j,m\}, \{k,m\} \in E,\label{uf1:c4}\\
& \eta^C_{ik}, \eta^D_{ik}, \nu_{ikjm} \geq 0,  &\forall \{i, k\}, \{j,m\}, \{k,m\} \in E,\nonumber\\
& z_{ik}, y_k \in \{0,1\},  &\forall \{i, k\} \in E, k \in N, \nonumber\\
& 0 \leq  r_k \leq R_k, &\forall k\in N,\nonumber
\end{align}

\noindent where constraints \eqref{uf1:c1}-\eqref{uf1:c4} are linearizations of the products $\d^C(a_i, x_{k})z_{ik}$, $\d^D(a_i, x_{k})z_{ik}$ and $\d^H(x_{k},x_{m})z_{ik}z_{jm}$, in which the constants $\hat{D}^C_{ik}$, $\hat{D}^C_{ik}$ and  $\hat{D}^H_{ik}$ are upper bounds on the distance from nodes $a_i$ to $x_k$ for the collection, distribution and transportation cost functions, respectively. These bounds depend on the distances used for the cost functions. In particular, for Euclidean distances, one can choose $\hat{D}^C_{ik} = \d^C(a_i,a_k) + 2R_k$, $\hat{D}^D_{ik} = \d^D(a_i,a_k) + 2R_k$ and $\hat{D}^H_{km} = \d^H(a_k,a_m) + R_k + R_m$.

Observe that the set of constraints \eqref{uf1:c4} can be alternatively rewritten as:
\begin{align}
\nu_{ikjm} \geq   \d^H(x_{k},x_{m}) - \hat{D}^H_{km}(1- z_{ik}), &&\forall \{i, k\}, \{j,m\}, \{k,m\} \in E,\label{uf1:c41}\\
\nu_{ikjm} \geq   \d^H(x_{k},x_{m}) - \hat{D}^H_{km}(1- z_{jm}), &&\forall \{i, k\}, \{j,m\}, \{k,m\} \in E.\label{uf1:c42}
\end{align}

However, the above formulation is not suitable to be solved \emph{directly} by any of the available MISOCO solvers since \eqref{uf1:c1} -\eqref{uf1:c4} are not rigorously speaking SOC constraints. Nevertheless, one can introduce adequate sets of nonnegative auxiliary variables for the distances, $\d^C(a_i, x_{k})$, $\d^D(x_k, a_i)$ and $\d^H(x_k,x_m)$ in those constraints, $d_{ik}^C, d_{ik}^D$ and $d_{km}^H$, respectively, for $\{i,k\},  \{k,m\} \in E$, and then rewrite \eqref{uf1:c1} -\eqref{uf1:c4} as:
\begin{align}
d_{ik}^C \geq \d^C(a_i, x_{k}), &\forall \{i,k\} \in E,\label{cc:1}\\
d_{ik}^D \geq \d^D(a_i, x_{k}), &\forall \{i,k\} \in E,\label{cc:2}\\
d_{km}^H \geq \d^H(x_k, x_m), &\forall \{i,k\} \in E,\label{cc:3}\\
\eta^C_{ik} \geq \d^C_{ik}- \hat{D}^C_{ik} (1-z_{ik}), &\forall \{i, k\} \in E,\label{cc:4}\\
\eta^D_{ik} \geq \d^D_{ik} - \hat{D}^D_{ik} (1-z_{ik}), &\forall \{i, k\} \in E,\label{cc:5}\\
\nu_{ikjm} \geq   \d^H_{km} - \hat{D}^H_{km}(2- z_{ik}-z_{jm}), &\forall \{i, k\}, \{j,m\}, \{k,m\} \in E,\label{cc:6}\\
d_{ik}^C , d_{ik}^D , d_{km}^H \geq 0,  &\forall \{i, k\}, \{k,m\} \in E.\nonumber
\end{align}
where now, \eqref{cc:1}-\eqref{cc:3} are SOC-representable (see e.g., \cite{BPE14}) and \eqref{cc:4}-\eqref{cc:6} are linear constraints.

The above formulation has $O(|E|^2)$ variables, $O(|E|^2)$  linear constraints and $O(|E|)$ non-linear constraints.

Note that in case $\d^C = \d^D$, a single set of $\eta$-variables, instead of $\eta^C$ and $\eta^D$, can be used to linearize the formulations.

In \cite{AV98} the authors propose a multicomodity-flow based formulation for the Uncapacitated Hub Location problem together with a branch and bound strategy for solving it. This formulation directly applies to our problem after introducing the following set of binary variables:
$$
p_{ijkm} = \left\{\begin{array}{cl}
1 & \mbox{if the arc $(k,m)$ is used to route the flow between $i$ and $j$,}\\
0 & \mbox{otherwise}.
\end{array}\right.
$$
for $\{k,m\}\in E$.

Thus inducing a formulation with $O(|E|^2)$ binary variables in contrast to the $O(|E|)$ binary variables in our formulation.

\subsection{Second Reformulation}\label{f2}

As in the previous formulation,  we consider, again, the variables $\eta^C_{ik}$ and $\eta_{ik}^D$ to represent the products of the collection/distribution distances and the allocation variables. Now, instead of the four-index variables $\nu_{ikjm}$, we consider two-index variables $\mu_{km}$, which represent the aggregated value of the overall cost flow on a given inter-hub arc $(k, m)$, i.e.,
$$
\mu_{km} =  \dsum_{i,j\in N:\atop \{i,k\}, \{j,m\} \in E}  w_{ij} \d^H(x_{k},x_{m}) z_{ik}z_{jm},
$$
for all $\{k, m\} \in E$.

The resulting equivalent reformulation is the following:

\begin{align}
\min & \dsum_{\{i,k\}\in E}  (O_i \eta^C_{ik} + D_i \eta^D_{ik}) + \dsum_{\{k, m\}\in E} \mu_{km} + \dsum_{k\in N} f_k y_k + \dsum_{k\in N} g_k(r_k) & \nonumber\\
\mbox{s.t. } & \eqref{uf0:c1}, \eqref{uf0:c2}, \eqref{uf0:c3}, \eqref{uf1:c1}, \eqref{uf1:c2}, \nonumber\\
& \mu_{km} \geq \dsum_{i,j\in N}  w_{ij} \d^H(x_{k},x_{m}) z_{ik}z_{jm}, & \forall \{k, m\} \in E,\label{uf2:exp0}\\
& \eta^C_{ik}, \eta^D_{ik} \geq 0,  &\forall \{i, k\}, \{j , m\} \in E,\nonumber\\
& z_{ik}, y_k \in \{0,1\}, & \forall \{i,k\} \in E, k \in N, \nonumber\\
& 0 \leq  r_k \leq R_k, & \forall k\in N,\nonumber\\
& \mu_{km} \geq 0, & \forall \{k, m\} \in E.\nonumber
\end{align}

\noindent Observe that, taking into account the non-negativity of variables and coefficients, and the minimization objective function, it is guaranteed that for optimal solutions equality will hold for constraints \eqref{uf2:exp0}, which establishes the objective function value for the routing of inter-hub flows.

Note also that the trilinear terms of constraints \eqref{uf2:exp0} make them difficult to handle algorithmically. Below we propose an equivalent set of linear constraints, which can be used for modeling the  values of variables $\mu_{km}$. First, in Proposition \ref{propo:1} we give a linearization of the bilinear products of $z$ variables.

\begin{prop}\label{propo:1}
The set of nonlinear constraints \eqref{uf2:exp0} is equivalent to the following set of inequalities:
\begin{equation}\label{uf2:exp}
\mu_{km} \geq \d^H(x_{k},x_{m})\,\dsum_{(i,j) \in S}  w_{ij} (z_{ik}+z_{jm}-1), \;\;  \forall S \subseteq N_{km}, \forall \{k, m\} \in E.
\end{equation}
where $N_{km}  = \{(i,j) \in N\times N: \{i,k\}, \{j,m\} \in E\}$.
\end{prop}

\begin{proof}

Let $(\bar z,\bar \mu) \in \{0,1\}^{|E|} \times \R^{|E|}_+$. Assume first that $(\bar z,\bar \mu)$ satisfies \eqref{uf2:exp0}. Let $\{k, m\} \in E$  and $S \subseteq N_{km}$. Then:
\begin{eqnarray*}
\begin{split}
\bar \mu_{km} &\stackrel{\eqref{uf2:exp0}}{\geq}   \d^H(x_{k},x_{m}) \dsum_{(i,j) \in N_{km}} w_{ij} \bar z_{ik} \bar z_{jm} = \d^H(x_{k},x_{m})\dsum_{(i,j)\in N_{km}:\atop \bar z_{ik}=\bar z_{jm}=1} w_{ij}(\bar z_{ik}+\bar z_{jm}-1)\\
& \geq  \d^H(x_{k},x_{m}) \dsum_{(i,j)\in S:\atop \bar z_{ik}=\bar z_{jm}=1} w_{ij} \bar z_{ik} \bar z_{jm}  \geq \d^H(x_{k},x_{m})\dsum_{(i,j)\in S} w_{ij} (\bar z_{ik}+ \bar z_{jm}-1).
\end{split}
\end{eqnarray*}

Hence, \eqref{uf2:exp} is verified.\\

On the other hand, if $(\bar z,\bar \mu)$ satisfies \eqref{uf2:exp}, for all $\{k, m\} \in E$ and $S \subseteq N_{km}$, then, in particular, chosing $\overline S=\{(i,j)\in N_{km}: \bar z_{ik}=\bar z_{km}=1\}$, we have that
\begin{eqnarray*}
\begin{split}
\bar \mu_{km} &\stackrel{\eqref{uf2:exp}}{\geq} \d^H(x_{k},x_{m}) \dsum_{(i,j) \in \overline S}  w_{ij} (\bar z_{ik}+\bar z_{jm}-1) =\d^H(x_{k},x_{m}) \dsum_{(i,j) \in \overline S}  w_{ij} \\
&= \d^H(x_{k},x_{m}) \dsum_{(i,j)\in N_{km}}  w_{ij} \bar z_{ik} \bar z_{jm},
\end{split}
\end{eqnarray*}
\noindent so \eqref{uf2:exp0} is also verified.
\end{proof}
Observe that the family \eqref{uf2:exp} may have an exponential number of constraints (as in complete hub-and-spoke networks), namely $|E| \times 2^{|E|}$ inequalities, which are nonlinear because of the products $\d^H(x_{k},x_{m}) \times z_{ik}$. Note that since constraint \eqref{uf2:exp} is only active in case the sets $S$ are adequately chosen as the whole pairs of users, $i$ and $j$ in $N$, linked to $k$ and $m$, respectively, for $\{k,m\} \in E$ (which are unknown), one cannot avoid the use of the exponentially many constraints.

As stated in the following result, those terms can be suitably linearized by introducing additional decision variables and applying McKormick transformation~\cite{McCormick76}.

\begin{prop}
Nonlinear constraints \eqref{uf2:exp} can be replaced by the following set of inequalities:
\begin{align}
\mu_{km} &\geq \!\!\!\!\dsum_{(i,j) \in S}  w_{ij} (\theta_{ikm}+\theta_{jmk}-\d^H(x_{k},x_{m})), \;\; \forall \{k, m\} \in E, \forall S \subseteq N_{km},\label{uf2:exp1}\\
\theta_{ikm} &\geq \d^H(x_{k},x_{m}) - \hat{D}^H_{km}(1-z_{ik}),  \forall \{i,k\}, \{k,m\} \in E, \label{uf2:Coll}\\
\theta_{ikm} &\geq 0,   \forall \{i,k\}, \{k,m\} \in E, \label{uf2:Dom}
\end{align}
where the new variables $\theta_{ikm}$ model the routing cost of the flow through the inter-hub arc $(k,m)$ with origin at $i$ and the coefficient $\hat{D}^H_{km}$ is an upper bound on the distance between the positions of hubs $k$ and $m$.
\end{prop}

This reformulation has $O(|N||E|)$ variables, $O(|E|)$ non-linear constraints and exponentially many linear constraints. Then, to take advantage of the smaller number of nonlinear constraints of this representation we propose an incomplete formulation, in which constraints  \eqref{uf2:exp} are added on-the-fly, embedded in a branch-and-cut scheme, as described in the next section.

\section{Branch-and-Cut Solution Algorithm}\label{sec:4}
In this section we develop a branch-and-cut solution algorithm for the UHLPN, based on the reformulation introduced in Section \ref{f2} where the routing costs of access and distribution flows are linearized via constraints \eqref{uf1:c1} and \eqref{uf1:c2}, and the costs of inter-hub flows are linearized using constraints  \eqref{uf2:exp1}-\eqref{uf2:Dom}. Given that the set of constraints \eqref{uf2:exp1} is of exponential size, only a small number of them is considered initially, and the remaining ones are handled as lazy constraints, so they are only separated at the nodes of the enumeration tree where a solution is found with binary values for the $z$. \\
As we show below, in such a case, separation can be carried out by inspection. Moreover, it is also possible to apply a more selective policy in which, instead of generating all violated constraints, only the most violated one is identified and added.

For all $S \subseteq N \times N$, we use the following notation in our approach:
\begin{itemize}
\item[] $S^+= \{i\in N: \exists j \in N \mbox{ with } (i,j) \in S\}$,
\item[] $S^+(i) = \{j \in N: (i,j) \in S\}$ for all $i\in S^+$,
\item[] $S^-= \{j\in N: \exists i \in N \mbox{ with } (i,j) \in S\}$,
\item[] $S^-(j) = \{i \in N: (i,j) \in S\}$ for all $j\in S^-$.
\item[] $O_i(S) = \!\!\!\!\!\!\dsum_{j\in S^-(i)} \!\!\!w_{ij}$ is the overall flow with origin $i$ and destination in $S$,  for all $i\in S^+$,
\item[] $D_j(S) = \!\!\!\!\!\!\dsum_{i \in S^+(j)} \!\!\!w_{ij}$ is the
overall flow with origin in $S$ and destination $j$, for all $j\in S^-$,
\item[] $w(S)=\!\!\!\dsum_{(i,j)\in S} w_{ij}$ is the overall flow over the connections in $S$.
\end{itemize}

With such a notation, for two potential hubs $k, m \in N$ with $\{k,m\}\in E$, and $S\subseteq N_{km}$, constraint \eqref{uf2:exp1} reads as:
\begin{align}
\mu_{km} \geq \dsum_{i \in S^+} O_i(S) \theta_{ikm} + \dsum_{j\in S^-} D_j(S) \theta_{jmk} - w(S)\d^H(x_{k},x_{m}). \label{uf2:exp1_1}
\end{align}

In Algorithm \ref{alg} we describe the pseudocode of the proposed solution scheme. There, we denote by ${\rm UHLPN}(\mathcal{S})$ the problem formulated as in Section \ref{f2} but where only the constraints  \eqref{uf2:exp} involving the sets in the pool $\mathcal{S} \subseteq \{S: S \subset N\times N\}$ are added.

\begin{algorithm}[h]

 \KwData{$\mathcal{S}=\emptyset$, \texttt{violation}=\texttt{true}}

\While{ {\rm \texttt{violation}=\texttt{true}}}{
 \texttt{violation}=\texttt{false}\\
Solve ${\rm UHLPN}(\mathcal{S})$: $(\bar z, \bar r, \bar\mu, \bar x)$.\\
\For{$\{k, m\} \in E$ with $\bar y_k=\bar y_m=1$}{
$S_{km} = \{(i,j) \in N_{km}: \bar z_{ik}=\bar z_{jm}=1\}$.\\
\If{$(\{k,m\}, S_{km})$ violates \eqref{uf2:exp}}{
	$\mathcal{S} \rightarrow \mathcal{S} \cup \{S_{km}\}$\\
	\texttt{violation}=\texttt{true}.
	}}
	}

 \caption{Branch-and-Cut Solution Approach.\label{alg}}
\end{algorithm}

Observe that to check the violation of constraint \eqref{uf2:exp1_1} for fixed $\{k, m\} \in E$ with $\bar y_k=\bar y_m=1$ for a given feasible solution $(\bar x,\bar z, \bar r; \bar \theta, \bar \mu)$,  we consider the following choice for the $S$-set:
$$
S = \{(i,j) \in N_{km}: \bar z_{ik} = \bar z_{jm}=1\}.
$$
Then, \eqref{uf2:exp1_1} is active only whenever $\dsum_{i \in S^+} O_i(S) \bar\theta_{ikm} = \dsum_{j\in S^-} D_j(S) \bar \theta_{jmk} = w(S)\d^H(\bar x_{k}, \bar x_{m})$ (in that case $\bar \mu_{km} = w(S)\d^H(\bar x_{k},\bar x_{m})$) or if $\bar x_k$ and $\bar x_m$ coincide. Otherwise, since $\bar \theta_{ikm} = \d^H(\bar x_{k},\bar x_{m}) \bar z_{ik}$ and $\bar \theta_{jmk} = \d^H(\bar x_{k},\bar x_{m}) \bar z_{jm}$ one would have that $\dsum_{i \in S^+} O_i(S) \neq \dsum_{j \in S^-} D_j(S)$ being then the overall  flow generated at the origins of $S$ different from the overall flow with destination at the destinations of $S$. Furthermore, in case $\dsum_{i \in S^+} O_i(S)= \dsum_{j \in S^-} D_j(S)$, this flow must coincide with the overall flow given by the origins and destinations in $S$.

Clearly, the constraint \eqref{uf2:exp1} with a maximum right-hand-side value for the solution $(\bar x,\bar z, \bar r; \bar \theta, \bar \mu)$,  is the one associated with set $S$. Hence, in order to solve the separation problem one only has to check whether or not the constraint \eqref{uf2:exp1} associated with $S$ is violated, or equivalently, whether or not:
$$
\bar \mu_{km} < \bar d_{km}  w(S),
$$
where $\bar d_{km}=\d^H(\bar  x_{k},\bar x_{m})$ is the routing cost of the inter-hub flows through arc $(\bar x_k, \bar x_m)$.

If the above condition is not met, then the constraint \eqref{uf2:exp1} associated with $S$ is violated, so the following cut is added:
$$
\mu_{km} \geq \dsum_{i \in S^+} O_i(S) \theta_{ikm} + \dsum_{j\in S^-} D_j(S) \theta_{jmk} - w(S)\d^H(x_{k},x_{m}). \label{uf2:exp1_2}
$$
One can also easily check for the most violated inequality, i.e., find $\{\bar k, \bar m\}$ in $\arg\dmax_{k, m \in N}$ $\Big\{\bar d_{km}w(S) - \bar \mu_{km} \Big\}$ with  $d_{\bar k\bar m}w(S) > \bar \mu_{\bar k \bar m}$.

\section{Experiments}\label{sec:5}

We have performed a series of experiments to test and compare the formulations. We have used the most common datasets in hub location: AP (Australia Post) \cite{Ernst96} and CAB (Civil Aeronautics Board) \cite{OKelly_EJOR87} with $n \in \{10, 20\}$ in $\R^2$ in which the network is complete. The coordinates, $a$, the OD flows, $w$, and the set-up costs, $f$, are taken from these instances. We have implemented the HLPN with distances $\d^C, \d^D$ and $\d^H$ induced by the $\ell_1$, $\ell_2$ and the $\ell_\infty$ norms such that $\d^H\leq\d^C=d^D$, i.e., the combinations of norms for $(\d^C= \d^D, \d^H)$ in $\{(\ell_p,\ell_q): p\leq q, p, q \in \{1,2,\infty\}\}$. We also apply an economy of scale factor $\alpha\in \{0.2, 0.5, 0.8, 1\}$. In case $\alpha=1$, we use only combinations of hub-to-hub distances and hub-to-spoke distances such that $\d^H < \d^D$, i.e., pair of norms in the form $(\ell_p, \ell_q)$ with $p>q$.

We consider unit-ball norm-based neighborhoods. In particular we choose as basic neighborhoods the sets $\{z\in \R^2: \|z\|_p \leq 1\}$ for $p\in \{1, 2, \infty\}$. Different upper bounds have been consider for the dilations of the neighborhoods. In particular we fix $R_k = \tau \;\min_{i\in N:\atop i\neq k} \d^C(a_i,a_k)$ for $k\in N$, with $\tau \in \{0.25, 0.5, 1, 1.5\}$. The radius-dependent set-up costs are linear and such that $f_{\rm set-up} (y,r) = \dsum_{k\in N} f_k y_k +  \dsum_{k\in N} \Lambda_k r_k$, where $\Lambda_k = \rho f_k$ with $\rho \in \{0.01, 0.1, 1, 2\}$.

With these choices, for each value of $n$ we solved $864$ instances. Thus, moving $n$ and considering the two datasets, we solved $5184$ instances, each of them with the  two reformulations. A time limit of $2$ hours was set for all the problems.

The two approaches were coded in Python, and solved using Gurobi  8.0 in a Mac OSX Mojave with an Intel Core i7 processor at 3.3 GHz and 16GB of RAM. For the branch-and-cut approach we use the default lazy callbacks implemented in Gurobi. We denote by $(F1)$ the solution approach based on solving the compact formulation described in Section \ref{f1} and by $(F2)$ the branch-and-cut approach proposed in Section \ref{f2}.

In tables \ref{t1:10} and \ref{t1:20} we report the average results of our computational experience in the following layout:
\begin{itemize}
\item The average  CPU time (in seconds) required for solving the problem with the two approaches: ${\rm Time}_1$ (for (F1)) and ${\rm Time}_2$ (for (F2).
\item The percentage of unsolved instances with the two approaches: ${\rm US}_1$ (for (F1)) and ${\rm US}_2$ (for (F2)).
\end{itemize}

From the results, one can observe that except in a few of the small instances, the branch-and-cut approach (F2) requires less CPU Time to solve the problems for polyhedral neighborhoods ($\ell_1$ and $\ell_\infty$ neighborhoods). For  disk-shaped neighborhoods, we note that the small instances are solved in smaller times using the compact approach (F1), but as the sizes of the instances increase, the branch-and-cut performs slightly better in CPU time and in many cases, problems that were not able to be solved within the time limit with the compact approach, are then  solved by the branch-and-cut method.  Similar conclusions are derived observing the percentage of unsolved instances.

In order to analyze the performance of the branch-and-cut approach, we show in Table \ref{t2} the number of cuts needed to solve the problem with such an algorithm. The first observation that comes from the results is that the number of cuts is small. Note that the complete formulation of the branch-and-cut approach requires an exponential (in $|E|=n^2$) number constraints. However, on average, the overall number of cuts is rather small ($103$ cuts for the $n=10$ instances and $401$ for the $n=20$ instances). Moreover, although the shape of the neighborhoods seems to affect the difficulty of the problem,  the number of cuts is similar for the three different neighborhoods considered in these experiments. Finally, we remark that the CAB dataset with $n=10$ nodes requires, significatively, much less cuts than the AP dataset for the same number of points.

\setlength{\tabcolsep}{2pt}

\begin{table}[H]
\begin{center}{
\begin{adjustbox}{angle=90}
\begin{tabular}{|c|c|c|c|c||rr|rr||rr|rr||rr|rr|}\cline{6-17}
       \multicolumn{5}{c|}{}  & \multicolumn{12}{c|}{Neighborhoods}\\\cline{6-17}
      \multicolumn{5}{c|}{}    & \multicolumn{4}{c|}{$\ell_1$}  & \multicolumn{4}{c|}{$\ell_2$}  & \multicolumn{4}{c|}{$\ell_\infty$} \\\hline
$n$ & Dataset & $\alpha$ & $\d^C$& $\d^H$ & Time$_1$ & Time$_2$ & US$_1$ & US$_2$ & Time$_1$ & Time$_2$ & US$_1$ & US$_2$ & Time$_1$ & Time$_2$ & US$_1$ & US$_2$ \\\hline
\multirow{24}{*}{10} & \multirow{12}[0]{*}{AP} & \multirow{3}{*}{0.2} & \multirow{2}{*}{$\ell_1$} & $\ell_2$ & 6.71  & 4.71  & 0\% & 0\% & 17.45 & 148.57 & 0\% & 0\% & 7.70  & 4.27  & 0\% & 0\% \\
      &       &       &       & $\ell_\infty$ & 5.80  & 2.79  & 0\% & 0\% & 9.25  & 65.21 & 0\% & 0\% & 6.10  & 2.23  & 0\% & 0\% \\\cline{4-17}
      &       &       & $\ell_2$ & $\ell_\infty$ & 5.09  & 5.41  & 0\% & 0\% & 12.00 & 142.01 & 0\% & 0\% & 4.31  & 2.66  & 0\% & 0\% \\\cline{3-17}
      &       & \multirow{3}{*}{0.5} & \multirow{2}{*}{$\ell_1$} & $\ell_2$ & 8.27  & 11.06 & 0\% & 0\% & 20.15 & 323.86 & 0\% & 0\% & 9.60  & 6.23  & 0\% & 0\% \\
      &       &       &       & $\ell_\infty$ & 8.23  & 3.73  & 0\% & 0\% & 11.12 & 83.29 & 0\% & 0\% & 8.19  & 2.87  & 0\% & 0\% \\
      &       &       & $\ell_2$ & $\ell_\infty$ & 5.72  & 7.45  & 0\% & 0\% & 13.45 & 197.20 & 0\% & 0\% & 6.30  & 4.72  & 0\% & 0\% \\\cline{3-17}
      &       & \multirow{3}{*}{0.8} & \multirow{2}{*}{$\ell_1$} & $\ell_2$ & 21.45 & 13.08 & 0\% & 0\% & 40.78 & 455.11 & 0\% & 0\% & 12.84 & 10.92 & 0\% & 0\% \\
      &       &       &       & $\ell_\infty$ & 10.68 & 6.11  & 0\% & 0\% & 17.14 & 122.06 & 0\% & 0\% & 10.23 & 5.37  & 0\% & 0\% \\\cline{4-17}
      &       &       & $\ell_2$ & $\ell_\infty$ & 6.90  & 14.35 & 0\% & 0\% & 21.65 & 259.07 & 0\% & 0\% & 7.22  & 8.19  & 0\% & 0\% \\\cline{3-17}
      &       & \multirow{3}{*}{1} & \multirow{2}{*}{$\ell_1$} & $\ell_2$ & 20.48 & 17.91 & 0\% & 0\% & 43.60 & 524.87 & 0\% & 0\% & 23.86 & 20.61 & 0\% & 0\% \\
      &       &       &       & $\ell_\infty$ & 10.67 & 7.10  & 0\% & 0\% & 18.53 & 127.09 & 0\% & 0\% & 10.70 & 5.71  & 0\% & 0\% \\\cline{4-17}
      &       &       & $\ell_2$ & $\ell_\infty$ & 8.38  & 14.56 & 0\% & 0\% & 23.10 & 467.01 & 0\% & 0\% & 8.24  & 13.34 & 0\% & 0\% \\\cline{2-17}
      & \multirow{12}{*}{CAB} & \multirow{3}{*}{0.2} & \multirow{2}{*}{$\ell_1$} & $\ell_2$ & 2.89  & 0.80  & 0\% & 0\% & 5.89  & 4.37  & 0\% & 0\% & 3.56  & 0.77  & 0\% & 0\% \\
      &       &       &       & $\ell_\infty$ & 2.14  & 0.91  & 0\% & 0\% & 4.62  & 4.38  & 0\% & 0\% & 2.11  & 0.88  & 0\% & 0\% \\\cline{4-17}
      &       &       & $\ell_2$ & $\ell_\infty$ & 1.64  & 0.96  & 0\% & 0\% & 2.79  & 3.89  & 0\% & 0\% & 1.70  & 1.19  & 0\% & 0\% \\\cline{3-17}
      &       & \multirow{3}{*}{0.5} & \multirow{2}{*}{$\ell_1$} & $\ell_2$ & 3.03  & 0.88  & 0\% & 0\% & 6.50  & 6.67  & 0\% & 0\% & 3.51  & 0.84  & 0\% & 0\% \\
      &       &       &       & $\ell_\infty$ & 2.77  & 0.96  & 0\% & 0\% & 4.80  & 5.77  & 0\% & 0\% & 2.12  & 0.93  & 0\% & 0\% \\\cline{4-17}
      &       &       & $\ell_2$ & $\ell_\infty$ & 1.75  & 0.92  & 0\% & 0\% & 2.90  & 4.12  & 0\% & 0\% & 1.75  & 1.38  & 0\% & 0\% \\\cline{3-17}
      &       & \multirow{3}{*}{0.8} & \multirow{2}{*}{$\ell_1$} & $\ell_2$ & 3.51  & 0.96  & 0\% & 0\% & 7.14  & 11.43 & 0\% & 0\% & 4.32  & 1.13  & 0\% & 0\% \\
      &       &       &       & $\ell_\infty$ & 3.34  & 1.67  & 0\% & 0\% & 5.63  & 7.75  & 0\% & 0\% & 3.33  & 1.26  & 0\% & 0\% \\\cline{4-17}
      &       &       & $\ell_2$ & $\ell_\infty$ & 2.04  & 1.01  & 0\% & 0\% & 3.01  & 4.06  & 0\% & 0\% & 2.38  & 1.55  & 0\% & 0\% \\\cline{3-17}
      &       & \multirow{3}{*}{1} & \multirow{2}{*}{$\ell_1$} & $\ell_2$ & 3.73  & 1.17  & 0\% & 0\% & 7.70  & 20.48 & 0\% & 0\% & 4.61  & 1.20  & 0\% & 0\% \\
      &       &       &       & $\ell_\infty$ & 3.51  & 1.92  & 0\% & 0\% & 5.59  & 7.67  & 0\% & 0\% & 4.45  & 1.30  & 0\% & 0\% \\\cline{4-17}
      &       &       & $\ell_2$ & $\ell_\infty$ & 2.83  & 1.26  & 0\% & 0\% & 3.55  & 5.07  & 0\% & 0\% & 2.26  & 1.52  & 0\% & 0\% \\\hline
\end{tabular}
\end{adjustbox}
}
\end{center}
\caption{Average CPU Times for $n=10$ instances.\label{t1:10}}
\end{table}

\begin{table}[H]
\begin{center}{
\begin{adjustbox}{angle=90}
\begin{tabular}{|c|c|c|c|c||rr|rr||rr|rr||rr|rr|}\cline{6-17}
       \multicolumn{5}{c|}{}  & \multicolumn{12}{c|}{Neighborhoods}\\\cline{6-17}
      \multicolumn{5}{c|}{}    & \multicolumn{4}{c|}{$\ell_1$}  & \multicolumn{4}{c|}{$\ell_2$}  & \multicolumn{4}{c|}{$\ell_\infty$} \\\hline
$n$ & Dataset & $\alpha$ & $\d^C$& $\d^H$ & Time$_1$ & Time$_2$ & US$_1$ & US$_2$ & Time$_1$ & Time$_2$ & US$_1$ & US$_2$ & Time$_1$ & Time$_2$ & US$_1$ & US$_2$ \\\hline
\multirow{24}{*}{20} & \multirow{12}{*}{AP } & \multirow{3}{*}{0.2} & \multirow{2}{*}{$\ell_1$} & $\ell_2$ & 2978.18 & 196.86 & 16.67\% & 0\% &\texttt{TL}& 2338.08 & 100\% & 0\% & 2949.46 & 153.83 & 16.67\% & 0\% \\
      &       &       &       & $\ell_\infty$ & 766.26 & 125.26 & 0\% & 0\% & 2987.55 & 1163.60 & 0\% & 0\% & 572.68 & 85.13 & 0\% & 0\% \\\cline{4-17}
      &       &       & $\ell_2$ & $\ell_\infty$ & 1633.13 & 343.89 & 0\% & 0\% & 1572.01 & 1070.30 & 0\% & 25.00\% & 1382.02 & 148.57 & 0\% & 0\% \\\cline{3-17}
      &       & \multirow{3}{*}{0.5} & \multirow{2}{*}{$\ell_1$} & $\ell_2$ & 1247.69 & 1061.37 & 66.67\% & 0\% &\texttt{TL}& 3567.22 & 100\% & 50\% & 2512.62 & 429.07 & 50\% & 0\% \\
      &       &       &       & $\ell_\infty$ & 1400.85 & 173.36 & 0\% & 0\% & 2286.23 & 2580.10 & 33.33\% & 0\% & 1034.32 & 136.85 & 0\% & 0\% \\\cline{4-17}
      &       &       & $\ell_2$ & $\ell_\infty$ & 1340.51 & 216.44 & 0\% & 0\% & 1715.79 & 1144.60 & 0\% & 25.00\% & 1084.91 & 178.72 & 0\% & 0\% \\\cline{3-17}
      &       & \multirow{3}{*}{0.8} & \multirow{2}{*}{$\ell_1$} & $\ell_2$ & 3048.17 & 949.88 & 50\% & 16.67\% & \texttt{TL} & 6823.31 & 100\% & 83.33\% & 1149.02 & 1224.31 & 66.67\% & 0\% \\
      &       &       &       & $\ell_\infty$ & 2172.49 & 475.21 & 0\% & 0\% & 4493.49 & 3378.25 & 33.33\% & 33.33\% & 1691.68 & 225.57 & 0\% & 0\% \\\cline{4-17}
      &       &       & $\ell_2$ & $\ell_\infty$ & 1418.60 & 646.56 & 0\% & 0\% & 2191.26 & 4331.19 & 0\% & 25.00\% & 1286.40 & 550.92 & 0\% & 0\% \\\cline{3-17}
      &       & \multirow{3}{*}{1} & \multirow{2}{*}{$\ell_1$} & $\ell_2$ & 1546.57 & 2109.69 & 66.67\% & 16.67\% &\texttt{TL}& 6589.77 & 100\% & 66.67\% & 2029.81 & 2174.79 & 66.67\% & 0\% \\
      &       &       &       & $\ell_\infty$ & 1953.51 & 876.26 & 16.67\% & 0\% & 4698.74 & 5476.91 & 33.33\% & 50\% & 2079.19 & 426.13 & 0\% & 0\% \\\cline{4-17}
      &       &       & $\ell_2$ & $\ell_\infty$ & 2003.45 & 683.06 & 0\% & 0\% & 2342.07 & 5375.71 & 0\% & 50\% & 1321.77 & 636.97 & 0\% & 0\% \\\cline{2-17}
      & \multirow{12}{*}{CAB } & \multirow{3}{*}{0.2} & \multirow{2}{*}{$\ell_1$} & $\ell_2$ & 3556.49 & 197.61 & 16.67\% & 0\% & 4794.58 & 929.98 & 50\% & 0\% & 4351.39 & 268.91 & 16.67\% & 0\% \\
      &       &       &       & $\ell_\infty$ & 358.94 & 101.48 & 0\% & 0\% & 1029.84 & 500.31 & 0\% & 0\% & 372.58 & 94.94 & 0\% & 0\% \\\cline{4-17}
      &       &       & $\ell_2$ & $\ell_\infty$ & 663.46 & 178.50 & 0\% & 0\% & 1061.48 & 599.50 & 0\% & 0\% & 389.84 & 141.71 & 0\% & 0\% \\\cline{3-17}
      &       & \multirow{3}{*}{0.5} & \multirow{2}{*}{$\ell_1$} & $\ell_2$ & 5379.52 & 973.99 & 16.67\% & 0\% & 6090.83 & 4266.86 & 83.33\% & 33.33\% & 4456.03 & 1171.41 & 66.67\% & 0\% \\
      &       &       &       & $\ell_\infty$ & 654.30 & 212.71 & 0\% & 0\% & 2180.23 & 1298.63 & 0\% & 0\% & 673.48 & 214.19 & 0\% & 0\% \\\cline{4-17}
      &       &       & $\ell_2$ & $\ell_\infty$ & 1212.39 & 495.69 & 0\% & 0\% & 2213.60 & 2203.75 & 0\% & 20\% & 690.35 & 348.92 & 0\% & 0\% \\\cline{3-17}
      &       & \multirow{3}{*}{0.8} & \multirow{2}{*}{$\ell_1$} & $\ell_2$ & 4778.30 & 1612.24 & 33.33\% & 16.67\% & 7045.19 & 6602.71 & 83.33\% & 83.33\% & 5796.86 & 2626.66 & 33.33\% & 16.67\% \\
      &       &       &       & $\ell_\infty$ & 1266.98 & 473.70 & 0\% & 0\% & 1619.34 & 3391.85 & 33.33\% & 0\% & 1382.59 & 523.06 & 0\% & 0\% \\\cline{4-17}
      &       &       & $\ell_2$ & $\ell_\infty$ & 1601.57 & 1233.34 & 0\% & 0\% & 2992.68 & 4878.26 & 20\% & 60\% & 1338.45 & 1157.61 & 0\% & 0\% \\\cline{3-17}
      &       & \multirow{3}{*}{1} & \multirow{2}{*}{$\ell_1$} & $\ell_2$ & 5349.04 & 2650.54 & 66.67\% & 50\% &\texttt{TL} &\texttt{TL} & 100\% & 100\% & 7143.93 & 2300.92 & 66.67\% & 33.33\% \\
      &       &       &       & $\ell_\infty$ & 1819.12 & 1111.10 & 0\% & 0\% & 2660.99 & 3567.64 & 33.33\% & 33.33\% & 1968.37 & 982.11 & 0\% & 0\% \\\cline{4-17}
      &       &       & $\ell_2$ & $\ell_\infty$ & 1759.34 & 2832.40 & 16.67\% & 0\% & 3100.26 & 6687.32 & 40\% & 80\% & 1926.63 & 2588.59 & 0\% & 0\% \\
\hline
\end{tabular}
\end{adjustbox}
}
\end{center}
\caption{Average CPU Times for $n=20$ instances.\label{t1:20}}
\end{table}

\begin{table}[H]
\begin{center}{
\begin{tabular}{|c|c|c||rrr||rrr||rrr||rrr|}\cline{4-15}
       \multicolumn{3}{c|}{}  & \multicolumn{6}{c|}{$n=10$} & \multicolumn{6}{c|}{$n=20$} \\\cline{4-15}
       \multicolumn{3}{c|}{}  & \multicolumn{3}{c|}{AP} & \multicolumn{3}{c|}{CAB} & \multicolumn{3}{c|}{AP} & \multicolumn{3}{c|}{CAB}\\\cline{4-15}
       \multicolumn{3}{c|}{}  & \multicolumn{12}{c|}{Neighborhoods}\\\hline
$\alpha$ & $\d^C$& $\d^H$ & {$\ell_1$}  & \ {$\ell_2$}  &  {$\ell_\infty$} & {$\ell_1$}  &  {$\ell_2$}  &  {$\ell_\infty$} & {$\ell_1$}  & \ {$\ell_2$}  &  {$\ell_\infty$} & {$\ell_1$}  &  {$\ell_2$}  &  {$\ell_\infty$}\\\hline
 \multirow{3}{*}{0.2} & \multirow{2}{*}{$\ell_1$} & $\ell_2$ & 	128   & 103   & 130&37    & 34    & 34 &  137   & 59    & 109 & 240   & 225   & 267 \\
       &       &        $\ell_\infty$ & 						134   & 185   & 135&34    & 49    & 33 &	132   & 140   & 116& 357   & 353   & 350 \\\cline{2-15}
       &        $\ell_2$ & $\ell_\infty$ & 					124   & 79    & 107 &17    & 30    & 26 &	78     & 44    & 60   &  259   & 160   & 220 \\\cline{1-15}
        \multirow{3}{*}{0.5} & \multirow{2}{*}{$\ell_1$} & $\ell_2$ & 		179   & 151   & 172&31    & 34    & 35 &		204   & 127   & 174&323   & 459   & 503 \\
              &       & $\ell_\infty$ &						156   & 226   & 158&36    & 50    & 39 &	208   & 263   & 211& 439   & 560   & 502 \\\cline{2-15}
              & $\ell_2$ & $\ell_\infty$ & 						194   & 130   & 166&17    & 30    & 27 &	163   & 75    & 131 & 322   & 418   & 382 \\\cline{1-15}
        \multirow{3}{*}{0.8} & \multirow{2}{*}{$\ell_1$} & $\ell_2$ & 			201   & 206   & 199&31    & 46    & 41 &		353   & 172   & 252& 637   & 486   & 785 \\
              &       & $\ell_\infty$ & 							200   & 289   & 190&44    & 61    & 54 &	354   & 435   & 241& 781   & 826   & 838 \\\cline{2-15}
              & $\ell_2$ & $\ell_\infty$ & 							226   & 174   & 195&19    & 27    & 29 &	254   & 150   & 196& 570   & 725   & 695 \\\cline{1-15}
        \multirow{3}{*}{1} & \multirow{2}{*}{$\ell_1$} & $\ell_2$ & 		201   & 194   & 215&36    & 49    & 40 &		376   & 141   & 327& 972   & 463   & 891 \\
              &       & $\ell_\infty$ & 								164   & 254   & 191&48    & 55    & 53 &	478   & 636   & 331&935   & 1324  & 1116 \\\cline{2-15}
              & $\ell_2$ & $\ell_\infty$ & 						257   & 235   & 205&18    & 26    & 24 &	272   & 208   & 256&772   & 598   & 931 \\\hline
\end{tabular}
}
\end{center}
\caption{Average number of cuts.\label{t2}}
\end{table}
\begin{table}[H]
\begin{tabular}{|c|c||c|c|c|}\cline{3-5}
\multicolumn{2}{c}{$\alpha=0.5$} & \multicolumn{3}{|c|}{Neighborhood}\\\hline
$\d^C$ & $\d^H$ & $\ell_1$ & $\ell_2$ & $\ell_\infty$\\\hline
$\ell_1$ & $\ell_1$  & \parbox[c]{3.8cm}{
\begin{tikzpicture}

\begin{axis}[
hide x axis,
hide y axis,
axis equal,
]
\path [draw=gray, fill=gray, opacity=0.3] (axis cs:20.355966023,16.167127237)
--(axis cs:20.355966023,16.167127237)
--(axis cs:20.355966023,16.167127237)
--cycle;

\path [draw=gray, fill=gray, opacity=0.3] (axis cs:37.30436959,32.079249435)
--(axis cs:37.30436959,32.079249435)
--(axis cs:37.30436959,32.079249435)
--cycle;

\path [draw=gray, fill=gray, opacity=0.3] (axis cs:35.606530375,38.835826089)
--(axis cs:24.512513179,49.929843285)
--(axis cs:13.418495983,38.835826089)
--(axis cs:24.512513179,27.741808893)
--cycle;

\addplot [semithick, black, dashed, forget plot]
table {%
39.98859202 19.773197847
37.30436959 19.773197847
};
\addplot [semithick, black, dashed, forget plot]
table {%
37.30436959 19.773197847
37.30436959 32.079249435
};
\addplot [semithick, black, dashed, forget plot]
table {%
23.572548971 31.529184022
37.30436959 31.529184022
};
\addplot [semithick, black, dashed, forget plot]
table {%
37.30436959 31.529184022
37.30436959 32.079249435
};
\addplot [semithick, black, dashed, forget plot]
table {%
40.867253869 38.565884488
37.30436959 38.565884488
};
\addplot [semithick, black, dashed, forget plot]
table {%
37.30436959 38.565884488
37.30436959 32.079249435
};
\addplot [semithick, black, dashed, forget plot]
table {%
27.520840478 46.921515986
27.520840478 46.921515986
};
\addplot [semithick, black, dashed, forget plot]
table {%
27.520840478 46.921515986
27.520840478 46.921515986
};
\addplot [semithick, black, dashed, forget plot]
table {%
36.067877874 44.894490748
27.520840478 44.894490748
};
\addplot [semithick, black, dashed, forget plot]
table {%
27.520840478 44.894490748
27.520840478 46.921515986
};
\addplot [semithick, black, dashed, forget plot]
table {%
19.345710923 51.97204094
27.520840478 51.97204094
};
\addplot [semithick, black, dashed, forget plot]
table {%
27.520840478 51.97204094
27.520840478 46.921515986
};
\addplot [semithick, black, dashed, forget plot]
table {%
33.635749233 49.663479404
27.520840478 49.663479404
};
\addplot [semithick, black, dashed, forget plot]
table {%
27.520840478 49.663479404
27.520840478 46.921515986
};
\addplot [semithick, gray, mark=triangle*, mark size=3, mark options={solid}, only marks, forget plot]
table {%
20.355966023 16.167127237
};
\addplot [semithick, gray, mark=triangle*, mark size=3, mark options={solid}, only marks, forget plot]
table {%
37.30436959 32.079249435
};

\addplot [semithick, gray, mark=triangle*, mark size=3, mark options={solid}, only marks, forget plot]
table {%
27.520840478 46.921515986
};

\addplot [semithick, black, mark=*, mark size=1, mark options={solid}, only marks, forget plot]
table {%
27.520840478 46.921515986
};

\addplot [semithick, black, mark=*, mark size=1, mark options={solid}, only marks, forget plot]
table {%
39.98859202 19.773197847
};
\addplot [semithick, black, mark=*, mark size=1, mark options={solid}, only marks, forget plot]
table {%
23.572548971 31.529184022
};
\addplot [semithick, black, mark=asterisk*, mark size=1, mark options={solid}, only marks, forget plot]
table {%
37.30436959 32.079249435
};
\addplot [semithick, black, mark=asterisk*, mark size=1, mark options={solid}, only marks, forget plot]
table {%
24.512513179 38.835826089
};
\addplot [semithick, black, mark=*, mark size=1, mark options={solid}, only marks, forget plot]
table {%
40.867253869 38.565884488
};
\addplot [semithick, black, mark=*, mark size=1, mark options={solid}, only marks, forget plot]
table {%
36.067877874 44.894490748
};
\addplot [semithick, black, mark=*, mark size=1, mark options={solid}, only marks, forget plot]
table {%
19.345710923 51.97204094
};
\addplot [semithick, black, mark=*, mark size=1, mark options={solid}, only marks, forget plot]
table {%
33.635749233 49.663479404
};
\end{axis}

\end{tikzpicture}} & \parbox[c]{3.8cm}{
\begin{tikzpicture}

\begin{axis}[
hide x axis,
hide y axis,
axis equal,
]
\draw[draw=gray,fill=gray,opacity=0.3] (axis cs:20.355966023,16.167127237) circle (0.0342646219393926);
\draw[draw=gray,fill=gray,opacity=0.3] (axis cs:37.30436959,32.079249435) circle (0.276805391044326);
\draw[draw=gray,fill=gray,opacity=0.3] (axis cs:24.512513179,38.835826089) circle (0.275460557853353);
\draw[draw=gray,fill=gray,opacity=0.3] (axis cs:33.635749233,49.663479404) circle (0.666525961787595);
\addplot [semithick, black, dashed, forget plot]
table {%
39.98859202 19.773197847
37.5806939022238 19.773197847
};
\addplot [semithick, black, dashed, forget plot]
table {%
37.5806939022238 19.773197847
37.5806939022238 32.0957312212387
};
\addplot [semithick, black, dashed, forget plot]
table {%
23.572548971 31.529184022
24.333302040013 31.529184022
};
\addplot [semithick, black, dashed, forget plot]
table {%
24.333302040013 31.529184022
24.333302040013 38.6266148456475
};
\addplot [semithick, black, dashed, forget plot]
table {%
40.867253869 38.565884488
37.5806939022238 38.565884488
};
\addplot [semithick, black, dashed, forget plot]
table {%
37.5806939022238 38.565884488
37.5806939022238 32.0957312212387
};
\addplot [semithick, black, dashed, forget plot]
table {%
27.520840478 46.921515986
33.3295508664064 46.921515986
};
\addplot [semithick, black, dashed, forget plot]
table {%
33.3295508664064 46.921515986
33.3295508664064 49.0714411113815
};
\addplot [semithick, black, dashed, forget plot]
table {%
36.067877874 44.894490748
33.3295508664064 44.894490748
};
\addplot [semithick, black, dashed, forget plot]
table {%
33.3295508664064 44.894490748
33.3295508664064 49.0714411113815
};
\addplot [semithick, black, dashed, forget plot]
table {%
19.345710923 51.97204094
33.3295508664064 51.97204094
};
\addplot [semithick, black, dashed, forget plot]
table {%
33.3295508664064 51.97204094
33.3295508664064 49.0714411113815
};
\addplot [semithick, gray, mark=triangle*, mark size=3, mark options={solid}, only marks, forget plot]
table {%
20.3801223539752 16.1914288247132
};
\addplot [semithick, gray, mark=triangle*, mark size=3,  mark options={solid}, only marks, forget plot]
table {%
37.5806939022238 32.0957312212387
};
\addplot [semithick,gray, mark=triangle*, mark size=3,  mark options={solid}, only marks, forget plot]
table {%
33.3295508664064 49.0714411113815
};
\addplot [semithick, gray, mark=triangle*, mark size=3,  mark options={solid}, only marks, forget plot]
table {%
24.333302040013 38.6266148456475
};
\addplot [semithick, black, mark=asterisk*, mark size=1, mark options={solid}, only marks, forget plot]
table {%
20.355966023 16.167127237
};
\addplot [semithick, black, mark=*, mark size=1, mark options={solid}, only marks, forget plot]
table {%
39.98859202 19.773197847
};
\addplot [semithick, black, mark=*, mark size=1, mark options={solid}, only marks, forget plot]
table {%
23.572548971 31.529184022
};
\addplot [semithick, black, mark=asterisk*, mark size=1, mark options={solid}, only marks, forget plot]
table {%
37.30436959 32.079249435
};
\addplot [semithick, black, mark=asterisk*, mark size=1, mark options={solid}, only marks, forget plot]
table {%
24.512513179 38.835826089
};
\addplot [semithick, black, mark=*, mark size=1, mark options={solid}, only marks, forget plot]
table {%
40.867253869 38.565884488
};
\addplot [semithick, black, mark=*, mark size=1, mark options={solid}, only marks, forget plot]
table {%
27.520840478 46.921515986
};
\addplot [semithick, black, mark=*, mark size=1, mark options={solid}, only marks, forget plot]
table {%
36.067877874 44.894490748
};
\addplot [semithick, black, mark=*, mark size=1, mark options={solid}, only marks, forget plot]
table {%
19.345710923 51.97204094
};
\addplot [semithick, black, mark=asterisk*, mark size=1, mark options={solid}, only marks, forget plot]
table {%
33.635749233 49.663479404
};
\end{axis}

\end{tikzpicture}} & \parbox[c]{4.5cm}{
\begin{tikzpicture}

\begin{axis}[
hide x axis,
hide y axis,
axis equal,
]
\path [draw=gray, fill=gray, opacity=0.3] (axis cs:20.355966023,16.167127237)
--(axis cs:20.355966023,16.167127237)
--(axis cs:20.355966023,16.167127237)
--cycle;

\path [draw=gray, fill=gray, opacity=0.3] (axis cs:37.30436959,32.079249435)
--(axis cs:37.30436959,32.079249435)
--(axis cs:37.30436959,32.079249435)
--cycle;

\path [draw=gray, fill=gray, opacity=0.3] (axis cs:27.520840478,60.147170495)
--(axis cs:27.520840478,43.796911385)
--(axis cs:11.170581368,43.796911385)
--(axis cs:11.170581368,60.147170495)
--cycle;

\addplot [semithick, black, dashed, forget plot]
table {%
39.98859202 19.773197847
37.30436959 19.773197847
};
\addplot [semithick, black, dashed, forget plot]
table {%
37.30436959 19.773197847
37.30436959 32.079249435
};
\addplot [semithick, black, dashed, forget plot]
table {%
23.572548971 31.529184022
37.30436959 31.529184022
};
\addplot [semithick, black, dashed, forget plot]
table {%
37.30436959 31.529184022
37.30436959 32.079249435
};
\addplot [semithick, black, dashed, forget plot]
table {%
24.512513179 38.835826089
27.520840478 38.835826089
};
\addplot [semithick, black, dashed, forget plot]
table {%
27.520840478 38.835826089
27.520840478 46.921515986
};
\addplot [semithick, black, dashed, forget plot]
table {%
40.867253869 38.565884488
37.30436959 38.565884488
};
\addplot [semithick, black, dashed, forget plot]
table {%
37.30436959 38.565884488
37.30436959 32.079249435
};
\addplot [semithick, black, dashed, forget plot]
table {%
27.520840478 46.921515986
27.520840478 46.921515986
};
\addplot [semithick, black, dashed, forget plot]
table {%
27.520840478 46.921515986
27.520840478 46.921515986
};
\addplot [semithick, black, dashed, forget plot]
table {%
36.067877874 44.894490748
27.520840478 44.894490748
};
\addplot [semithick, black, dashed, forget plot]
table {%
27.520840478 44.894490748
27.520840478 46.921515986
};
\addplot [semithick, black, dashed, forget plot]
table {%
33.635749233 49.663479404
27.520840478 49.663479404
};
\addplot [semithick, black, dashed, forget plot]
table {%
27.520840478 49.663479404
27.520840478 46.921515986
};
\addplot [semithick, gray, mark=triangle*, mark size=3, mark options={solid}, only marks, forget plot]
table {%
20.355966023 16.167127237
};
\addplot [semithick, gray, mark=triangle*, mark size=3, mark options={solid}, only marks, forget plot]
table {%
37.30436959 32.079249435
};
\addplot [semithick, gray, mark=triangle*, mark size=3, mark options={solid}, only marks, forget plot]
table {%
27.520840478 46.921515986
};
\addplot [semithick, black, mark=*, mark size=1, mark options={solid}, only marks, forget plot]
table {%
27.520840478 46.921515986
};
\addplot [semithick, black, mark=*, mark size=1, mark options={solid}, only marks, forget plot]
table {%
39.98859202 19.773197847
};
\addplot [semithick, black, mark=*, mark size=1, mark options={solid}, only marks, forget plot]
table {%
23.572548971 31.529184022
};
\addplot [semithick, black, mark=*, mark size=1, mark options={solid}, only marks, forget plot]
table {%
24.512513179 38.835826089
};
\addplot [semithick, black, mark=*, mark size=1, mark options={solid}, only marks, forget plot]
table {%
40.867253869 38.565884488
};
\addplot [semithick, black, mark=*, mark size=1, mark options={solid}, only marks, forget plot]
table {%
36.067877874 44.894490748
};
\addplot [semithick, black, mark=*, mark size=1, mark options={solid}, only marks, forget plot]
table {%
33.635749233 49.663479404
};
\end{axis}

\end{tikzpicture}}\\\hline
 $\ell_2$  &  $\ell_2$ & \parbox[c]{3.8cm}{
\begin{tikzpicture}

\begin{axis}[
hide x axis,
hide y axis,
axis equal,
]
\path [draw=gray, fill=gray, opacity=0.3] (axis cs:37.30436959,32.079249435)
--(axis cs:37.30436959,32.079249435)
--(axis cs:37.30436959,32.079249435)
--cycle;

\path [draw=gray, fill=gray, opacity=0.3] (axis cs:35.5627955177932,38.835826089)
--(axis cs:24.512513179,49.8861084277932)
--(axis cs:13.4622308402068,38.835826089)
--(axis cs:24.512513179,27.7855437502068)
--cycle;

\addplot [semithick, black, dashed, forget plot]
table {%
20.355966023 16.167127237
37.30436959 32.079249435
};
\addplot [semithick, black, dashed, forget plot]
table {%
39.98859202 19.773197847
37.30436959 32.079249435
};
\addplot [semithick, black, dashed, forget plot]
table {%
23.572548971 31.529184022
37.30436959 32.079249435
};
\addplot [semithick, black, dashed, forget plot]
table {%
40.867253869 38.565884488
37.30436959 32.079249435
};
\addplot [semithick, black, dashed, forget plot]
table {%
27.520840478 46.921515986
27.5107591827101 46.887862424083
};
\addplot [semithick, black, dashed, forget plot]
table {%
36.067877874 44.894490748
27.5107591827101 46.887862424083
};
\addplot [semithick, black, dashed, forget plot]
table {%
19.345710923 51.97204094
27.5107591827101 46.887862424083
};
\addplot [semithick, black, dashed, forget plot]
table {%
33.635749233 49.663479404
27.5107591827101 46.887862424083
};
\addplot [semithick, gray, mark=triangle*, mark size=3, mark options={solid}, only marks, forget plot]
table {%
37.30436959 32.079249435
};
\addplot [semithick, gray, mark=triangle*, mark size=3, mark options={solid}, only marks, forget plot]
table {%
27.5107591827101 46.887862424083
};
\addplot [semithick, black, mark=*, mark size=1, mark options={solid}, only marks, forget plot]
table {%
20.355966023 16.167127237
};
\addplot [semithick, black, mark=*, mark size=1, mark options={solid}, only marks, forget plot]
table {%
39.98859202 19.773197847
};
\addplot [semithick, black, mark=*, mark size=1, mark options={solid}, only marks, forget plot]
table {%
23.572548971 31.529184022
};
\addplot [semithick, black, mark=asterisk*, mark size=1, mark options={solid}, only marks, forget plot]
table {%
37.30436959 32.079249435
};
\addplot [semithick, black, mark=asterisk*, mark size=1, mark options={solid}, only marks, forget plot]
table {%
24.512513179 38.835826089
};
\addplot [semithick, black, mark=*, mark size=1, mark options={solid}, only marks, forget plot]
table {%
40.867253869 38.565884488
};
\addplot [semithick, black, mark=*, mark size=1, mark options={solid}, only marks, forget plot]
table {%
27.520840478 46.921515986
};
\addplot [semithick, black, mark=*, mark size=1, mark options={solid}, only marks, forget plot]
table {%
36.067877874 44.894490748
};
\addplot [semithick, black, mark=*, mark size=1, mark options={solid}, only marks, forget plot]
table {%
19.345710923 51.97204094
};
\addplot [semithick, black, mark=*, mark size=1, mark options={solid}, only marks, forget plot]
table {%
33.635749233 49.663479404
};
\end{axis}

\end{tikzpicture}} & \parbox[c]{3.8cm}{
\begin{tikzpicture}

\begin{axis}[
hide x axis,
hide y axis,
axis equal,
]
\draw[draw=gray,fill=gray,opacity=0.3] (axis cs:37.30436959,32.079249435) circle (0.0790579991593904);
\draw[draw=gray,fill=gray,opacity=0.3] (axis cs:24.512513179,38.835826089) circle (0.402576705440978);
\draw[draw=gray,fill=gray,opacity=0.3] (axis cs:33.635749233,49.663479404) circle (0.531277113425545);
\addplot [semithick, black, dashed, forget plot]
table {%
20.355966023 16.167127237
24.4977899304393 38.4335079258614
};
\addplot [semithick, black, dashed, forget plot]
table {%
39.98859202 19.773197847
37.3834275891594 32.0819062195949
};
\addplot [semithick, black, dashed, forget plot]
table {%
23.572548971 31.529184022
24.4977899304393 38.4335079258614
};
\addplot [semithick, black, dashed, forget plot]
table {%
40.867253869 38.565884488
37.3834275891594 32.0819062195949
};
\addplot [semithick, black, dashed, forget plot]
table {%
27.520840478 46.921515986
33.2570269379745 49.2908782075396
};
\addplot [semithick, black, dashed, forget plot]
table {%
36.067877874 44.894490748
33.2570269379745 49.2908782075396
};
\addplot [semithick, black, dashed, forget plot]
table {%
19.345710923 51.97204094
33.2570269379745 49.2908782075396
};
\addplot [semithick, gray, mark=triangle*, mark size=3, mark options={solid}, only marks, forget plot]
table {%
37.3834275891594 32.0819062195949
};
\addplot [semithick, gray, mark=triangle*, mark size=3, mark options={solid}, only marks, forget plot]
table {%
24.4977899304393 38.4335079258614
};
\addplot [semithick, gray, mark=triangle*, mark size=3, mark options={solid}, only marks, forget plot]
table {%
37.3834275891594 32.0819062195949
};
\addplot [semithick, gray, mark=triangle*, mark size=3, mark options={solid}, only marks, forget plot]
table {%
33.2570269379745 49.2908782075396
};
\addplot [semithick, gray, mark=triangle*, mark size=3, mark options={solid}, only marks, forget plot]
table {%
24.4977899304393 38.4335079258614
};
\addplot [semithick, gray, mark=triangle*, mark size=3, mark options={solid}, only marks, forget plot]
table {%
33.2570269379745 49.2908782075396
};
\addplot [semithick, black, mark=*, mark size=1, mark options={solid}, only marks, forget plot]
table {%
20.355966023 16.167127237
};
\addplot [semithick, black, mark=*, mark size=1, mark options={solid}, only marks, forget plot]
table {%
39.98859202 19.773197847
};
\addplot [semithick, black, mark=*, mark size=1, mark options={solid}, only marks, forget plot]
table {%
23.572548971 31.529184022
};
\addplot [semithick, black, mark=asterisk*, mark size=1, mark options={solid}, only marks, forget plot]
table {%
37.30436959 32.079249435
};
\addplot [semithick, black, mark=asterisk*, mark size=1, mark options={solid}, only marks, forget plot]
table {%
24.512513179 38.835826089
};
\addplot [semithick, black, mark=*, mark size=1, mark options={solid}, only marks, forget plot]
table {%
40.867253869 38.565884488
};
\addplot [semithick, black, mark=*, mark size=1, mark options={solid}, only marks, forget plot]
table {%
27.520840478 46.921515986
};
\addplot [semithick, black, mark=*, mark size=1, mark options={solid}, only marks, forget plot]
table {%
36.067877874 44.894490748
};
\addplot [semithick, black, mark=*, mark size=1, mark options={solid}, only marks, forget plot]
table {%
19.345710923 51.97204094
};
\addplot [semithick, black, mark=asterisk*, mark size=1, mark options={solid}, only marks, forget plot]
table {%
33.635749233 49.663479404
};
\end{axis}

\end{tikzpicture}} & \parbox[c]{4cm}{
\begin{tikzpicture}

\begin{axis}[
hide x axis,
hide y axis,
axis equal,
]
\path [draw=gray, fill=gray, opacity=0.3] (axis cs:37.5346076388758,32.3094874838758)
--(axis cs:37.5346076388758,31.8490113861242)
--(axis cs:37.0741315411242,31.8490113861242)
--(axis cs:37.0741315411242,32.3094874838758)
--cycle;

\path [draw=gray, fill=gray, opacity=0.3] (axis cs:32.4798489784739,46.8031618884739)
--(axis cs:32.4798489784739,30.8684902895261)
--(axis cs:16.5451773795261,30.8684902895261)
--(axis cs:16.5451773795261,46.8031618884739)
--cycle;

\addplot [semithick, black, dashed, forget plot]
table {%
20.355966023 16.167127237
37.0741315411242 31.8490113861242
};
\addplot [semithick, black, dashed, forget plot]
table {%
39.98859202 19.773197847
37.0741315411242 31.8490113861242
};
\addplot [semithick, black, dashed, forget plot]
table {%
23.572548971 31.529184022
37.0741315411242 31.8490113861242
};
\addplot [semithick, black, dashed, forget plot]
table {%
40.867253869 38.565884488
37.0741315411242 31.8490113861242
};
\addplot [semithick, black, dashed, forget plot]
table {%
27.520840478 46.921515986
27.6413543689997 46.8031618884739
};
\addplot [semithick, black, dashed, forget plot]
table {%
36.067877874 44.894490748
27.6413543689997 46.8031618884739
};
\addplot [semithick, black, dashed, forget plot]
table {%
19.345710923 51.97204094
27.6413543689997 46.8031618884739
};
\addplot [semithick, black, dashed, forget plot]
table {%
33.635749233 49.663479404
27.6413543689997 46.8031618884739
};
\addplot [semithick, gray, mark=triangle*, mark size=3, mark options={solid}, only marks, forget plot]
table {%
37.0741315411242 31.8490113861242
};
\addplot [semithick, gray, mark=triangle*, mark size=3, mark options={solid}, only marks, forget plot]
table {%
27.6413543689997 46.8031618884739
};
\addplot [semithick, black, mark=*, mark size=1, mark options={solid}, only marks, forget plot]
table {%
20.355966023 16.167127237
};
\addplot [semithick, black, mark=*, mark size=1, mark options={solid}, only marks, forget plot]
table {%
39.98859202 19.773197847
};
\addplot [semithick, black, mark=*, mark size=1, mark options={solid}, only marks, forget plot]
table {%
23.572548971 31.529184022
};
\addplot [semithick, black, mark=asterisk*, mark size=1, mark options={solid}, only marks, forget plot]
table {%
37.30436959 32.079249435
};
\addplot [semithick, black, mark=asterisk*, mark size=1, mark options={solid}, only marks, forget plot]
table {%
24.512513179 38.835826089
};
\addplot [semithick, black, mark=*, mark size=1, mark options={solid}, only marks, forget plot]
table {%
40.867253869 38.565884488
};
\addplot [semithick, black, mark=*, mark size=1, mark options={solid}, only marks, forget plot]
table {%
27.520840478 46.921515986
};
\addplot [semithick, black, mark=*, mark size=1, mark options={solid}, only marks, forget plot]
table {%
36.067877874 44.894490748
};
\addplot [semithick, black, mark=*, mark size=1, mark options={solid}, only marks, forget plot]
table {%
19.345710923 51.97204094
};
\addplot [semithick, black, mark=*, mark size=1, mark options={solid}, only marks, forget plot]
table {%
33.635749233 49.663479404
};
\end{axis}

\end{tikzpicture}}\\\hline
\end{tabular}
\caption{Solutions for the same AP ($n=10$) instance and different costs and neighborhood shapes.\label{fig:draws}}
\end{table}

Finally, in Figure \ref{fig:draws} we draw some of the solution for the AP dataset for $n=10$ and different choices for the costs $\d^C$ and $\d^H$ and neighborhoods shapes. Observe that both the shapes of the neighborhoods and the distance-based costs affect the optimal position of the hubs and the allocation pattern of the hub-and-spoke network.

\section{Conclusions}\label{sec:6}

We analyze in this paper a new version of the uncapacitated hub location problem with fixed costs where geographical flexibility is allowed for locating the hub nodes.  We propose a general framework for the problem by modeling the geographical flexibility using neighborhood-second-order-cone representable regions around the original positions of the nodes and measuring distribution, collection and transportation costs by means of $\ell_p$-norm distances. We propose a mixed integer non linear programming formulation for the problem and we provide two different formulations which make it tractable using commercial MISOCO solvers. The first formulation consists of a linearization of some bilinear and trilinear terms while the second one is based on introducing a novel set of exponentially many constraints for which an efficient separation oracle and a branch-and-cut approach is presented.

Future research on the topic includes the incorporation of preferences onto the neighborhoods allowing the decision-maker to determine his most favorite regions. Also, some different models of hub location  as those with incomplete backbone networks or covering objective functions can be extended to be analyzed with neighborhoods.

\section*{Acknowledgements}

The authors were partially supported by project  MTM2016-74983-C2-1-R (MINECO, Spain). We also would like to acknowledge Elena Fern\'andez (Universidad de C\'adiz) for her useful and detailed comments on previous versions of this manuscript.

\end{document}